\documentclass{birkau}

\usepackage{amsfonts,amssymb,amsmath,amsthm}
\usepackage{mathtools}
\usepackage{tikz}
\usepackage{comment,enumerate}
\usepackage{hyperref}
\hypersetup{colorlinks,citecolor=blue,filecolor=black, linkcolor=blue, urlcolor=blue}

\newlength{\auxlength}
\newlength{\auxlengthtwo}
\newlength{\auxlengththree}

\newcommand{\alg}{\mathbf}

\newcommand{\logic}{\mathcal}

\newcommand{\truthvalue}{\mathsf}

\newcommand{\assign}{:=}
\newcommand{\equals}{\approx}

\newcommand{\pair}[2]{\langle #1, #2 \rangle}

\newcommand{\tuple}{\overline}

\newcommand{\clone}{\mathsf}
\newcommand{\clonegen}[1]{\mathrm{Clo}(#1)}

\newcommand{\DMAClone}{\clone{DMA}}
\newcommand{\DLatClone}{\clone{DLat}}
\newcommand{\BiLatClone}{\clone{BiLat}}

\newcommand{\idntof}{\mathop{\idmap_{\Neither\mapsto\False}}}
\newcommand{\idntot}{\mathop{\idmap_{\Neither\mapsto\True}}}
\newcommand{\idbtot}{\mathop{\idmap_{\Both\mapsto\True}}}
\newcommand{\idntob}{\mathop{\idmap_{\Neither\mapsto\Both}}}
\newcommand{\idbton}{\mathop{\idmap_{\Both\mapsto\Neither}}}

\newcommand{\Truetton}{\mathop{\True_{\True\mapsto\Neither}}}
\newcommand{\Truettob}{\mathop{\True_{\True\mapsto\Both}}}
\newcommand{\Truenton}{\mathop{\True_{\Neither\mapsto\Neither}}}
\newcommand{\Truebtob}{\mathop{\True_{\Both\mapsto\Both}}}

\newcommand{\Deltanbpair}{\Delta_{\pair{\Neither}{\Both}}}

\newcommand{\pbp}{\mathrm{pbp}^{2}}
\newcommand{\nh}{\mathrm{nh}^{2}}
\newcommand{\mnh}{\mathrm{mnh}^{2}}
\newcommand{\np}{\mathrm{np}^{2}}
\newcommand{\mnp}{\mathrm{mnp}^{2}}
\newcommand{\mhnp}{\mathrm{mhnp}^{2}}
\newcommand{\mhnpt}{\mathrm{mhnp}^{3}}
\newcommand{\mnpt}{\mathrm{mnp}^{3}}
\newcommand{\npt}{\mathrm{np}^{3}}

\newcommand{\imp}{\rightarrow}
\newcommand{\imptwo}{\supset}

\newcommand{\biimp}{\leftrightarrow}

\mathchardef\mathhyphen="2D
\newcommand{\imptmin}{\leftrightarrow_{\mathrm{t\mathhyphen min}}}
\newcommand{\imptmax}{\rightarrow_{\mathrm{t\mathhyphen max}}}
\newcommand{\impimin}{\leftrightarrow_{\mathrm{i\mathhyphen min}}}
\newcommand{\impimax}{\rightarrow_{\mathrm{i\mathhyphen max}}}
\newcommand{\impG}{\rightarrow_{\mathrm{G}}}
\newcommand{\impcrisp}{\rightarrow_{\True\False}}
\newcommand{\equivcrisp}{\leftrightarrow_{\True\False}}

\newcommand{\Btwo}{\mathrm{B}_{\mathrm{2}}}
\newcommand{\Kthree}{\mathrm{K}_{\mathrm{3}}}
\newcommand{\Pthree}{\mathrm{P}_{\mathrm{3}}}
\newcommand{\DMfour}{\mathrm{DM}_{\mathrm{4}}}
\newcommand{\Btwoalg}{\alg{B}_{\alg{2}}}
\newcommand{\Kthreealg}{\alg{K}_{\alg{3}}}
\newcommand{\Pthreealg}{\alg{P}_{\alg{3}}}
\newcommand{\DMfouralg}{\alg{DM}_{\alg{4}}}

\newcommand{\True}{\mathsf{t}}
\newcommand{\False}{\mathsf{f}}
\newcommand{\Neither}{\mathsf{n}}
\newcommand{\Both}{\mathsf{b}}

\newcommand{\dmneg}{\mathord{-}}
\newcommand{\dual}{\partial}

\newcommand{\infleq}{\sqsubseteq}
\newcommand{\infgeq}{\sqsupseteq}
\newcommand{\ninfleq}{\not\infleq}
\newcommand{\ninfgeq}{\not\infgeq}
\newcommand{\inflneq}{\sqsubset}
\newcommand{\infgneq}{\sqsupset}
\newcommand{\infwedge}{\otimes}
\newcommand{\infvee}{\oplus}

\newcommand{\Leibniz}[2]{\Omega^{#1}(#2)}
\newcommand{\idmap}{\mathrm{id}}
\newcommand{\interdash}{\dashv \vdash}

\newcommand{\theoremstartswithitemize}{~}

\theoremstyle{plain}
\newtheorem{theorem}{Theorem}[section]
\newtheorem{proposition}[theorem]{Proposition}
\newtheorem{lemma}[theorem]{Lemma}
\newtheorem{fact}[theorem]{Fact}
\newtheorem{corollary}[theorem]{Corollary}

\theoremstyle{definition}
\newtheorem{definition}[theorem]{Definition}

\author[A. P\v{r}enosil]{Adam P\v{r}enosil}
\title{De Morgan clones \mbox{and four-valued logics}}

\address{Department of Mathematics\\
Vanderbilt University\\Tennessee 37235\\USA}
\urladdr{https://sites.google.com/site/adamprenosil}
\email{adam.prenosil@vanderbilt.edu}


\begin{document}

\begin{abstract}
  We study clones on a four-element set related to the clone $\DMAClone$ of all term functions of the sub\-directly irreducible four-element De~Morgan algebra $\DMfouralg$. We find generating sets for the clones of all functions preserving the subalgebras of $\DMfouralg$, the auto\-morphisms of~$\DMfouralg$, the truth order and the information order on $\DMfouralg$, as well as clones defined by conjunctions of these conditions. We identify the covers of $\DMAClone$ in the lattice of four-valued clones and describe the lattice of clones above $\DMAClone$ which contain the discriminator function. Finally, observing that each clone above $\DMAClone$ defines an expansion of the four-valued Belnap--Dunn logic, we classify these clones by their metalogical properties, specifically by their position within the Leibniz and Frege hierarchies of abstract algebraic logic.
\end{abstract}

\maketitle

\section{Introduction}

  The four truth values $\True$ (True), $\False$ (False), $\Neither$ (Neither True nor False), and $\Both$ (Both True and False) form the underlying set of the algebraic semantics for various non-classical logics which endorse the twin ideas of paraconsistency (some propositions may be both true and false) and paracompleteness (some propositions may be neither true nor false). The simplest and best-known of these logics is the four-valued Belnap--Dunn logic~\cite{dunn76,belnap77a,belnap77b}, also known as the logic of first-degree entailment. Its signature consists of a conjunction, disjunction, negation, and possibly also constants for truth and falsity. Many other logics expand of this logic by additional connectives~\cite{arieli+avron96,font+rius00,bou+rivieccio11,sano+omori14,de+omori15}.

  The aim of this paper is to (at least partly) map the expansions of Belnap--Dunn logic in a systematic way. To this end we adopt the signature-independent perspective of clone theory~\cite{clones14}. That is, we identify logics which differ in terms of their choice of primitive connectives but where the same definable connectives can be expressed. This approach enables us to impose some order on the jungle of all possible expansions of Belnap--Dunn logic.

  For example, consider the unary operator $\Box x$ such that $\Box \True = \True$ and $\Box a = \False$ otherwise, and the binary connective $x \impG y$ such that $a \impG b = \True$ if $a \leq b$ and $a \impG b = b$ otherwise. From a fine-grained point of view these yield two different expansions Belnap--Dunn logic. However, from a coarse-grained point of view these expansions can be identified, since the two operations are interdefinable using the operations of Belnap--Dunn logic.

  Here and throughout the paper, by an \emph{expansion} of Belnap--Dunn logic we mean a logic defined semantically by an expansion of (i.e.\ by adding additional connectives to) the four-element logical matrix described below which defines Belnap--Dunn logic. Let us emphasize that this definition is more restrictive than the one commonly found in the literature, which would merely require that Belnap--Dunn logic be precisely the fragment obtained by restricting the expansion to the appropriate signature.

  The typical meta\-logical properties studied by algebraic logicians -- such as the deduction theorem~\cite{blok+pigozzi01}, the~proof by cases property~\cite{czelakowski01, cintula+noguera13}, proto\-algebraicity~\cite{czelakowski01}, truth-equationality~\cite{raftery06}, and self\-extensionality~\cite{wojcicki79,font16} -- depend only on the definable connectives. A signature-independent approach is therefore entirely appropriate when trying to classify the expansions of Belnap--Dunn logic according to such metalogical properties.

  The algebra of truth-values of Belnap-Dunn logic is the sub\-directly irreducible four-element De~Morgan algebra $\DMfouralg$ over the universe $\DMfour \assign \{ \True, \False, \Neither, \Both \}$. Clones on this set will be called De~Morgan clones, in particular the clone of all term functions of this algebra will be denoted~$\DMAClone$. 

  We~first consider various De~Morgan clones related to the algebra $\DMfouralg$. In the first part of the paper (Section~\ref{sec: structure preserving}), we provide finite generating sets for all De~Morgan clones defined by pre\-serving some combination of the following: the sub\-algebras of~$\DMfouralg$, the auto\-morphisms of~$\DMfouralg$, the truth (lattice) order on~$\DMfouralg$, and the information order on~$\DMfouralg$.

  Here we follow in the footsteps of Arieli and Avron, who already established several results of this form~\cite{arieli+avron98,avron99,arieli+avron17} (see also~\cite[Section~5.3]{avron+arieli+zamansky18}). However, our method of proof and choice of generating sets differ from theirs. Since we intend to provide a self-contained presentation, we include proofs even for results which were already proved in the papers of Arieli and Avron. (This concerns, in our terminology, the clones of persistent functions and functions preserving $\Btwo$, $\Kthree$, $\Pthree$, and both $\Kthree$ and $\Pthree$.)

  Having described the De~Morgan clones which preserve some of the structure of~$\DMfouralg$, we then turn our attention to clones above $\DMAClone$ which \emph{fail} to preserve some of this structure (Section~\ref{sec: above dma}). As a simple example, a clone above $\DMAClone$ fails to preserve $\{ \True, \False \}$ if and only if it contains one of the constants $\Neither$ or $\Both$. As a consequence, we show that $\DMAClone$ has exactly three covers in the lattice of all four-valued clones. To this end it will suffice to describe the minimal clones above $\DMAClone$ which fail to preserve either the information order or the automorphisms of $\DMfouralg$, since $\DMAClone$ is the clone of all functions which preserve the information order and the automorphisms of $\DMfouralg$. These results will also help us describe the lattice of all De~Morgan clones above $\DMAClone$ which contain the discriminator function (Section~\ref{sec: discriminator}).

  Two remarks are in order here. Firstly, there is little hope of obtaining an informative description of the lattice of \emph{all} De~Morgan clones. While the lattice of all clones on a two-element set is countable and has a fairly simple structure first described by Post~\cite{post41}, already the lattice of clones on a three-element set has the cardinality of the continuum and its structure is widely assumed not to permit a tractable description~\cite{kerkhoff+poschel+schneider14}.

  Secondly, there are in fact only finitely many De~Morgan clones lying above~$\DMAClone$. The Baker--Pixley theorem~\cite{baker+pixley75} states that if the ternary majority function is a term function of a finite algebra $\alg{A}$, then the term functions of~$\alg{A}$ are precisely the functions which preserve the subalgebras of $\alg{A} \times \alg{A}$ (viewed as binary relations). In particular, each clone above $\DMAClone$ is the clone of all functions preserving some set of subalgebras of $\DMfouralg \times \DMfouralg$. The lattice of all De~Morgan clones above $\DMAClone$ can therefore in theory be obtained in a completely mechanical manner. In practice, however, such an approach does not seem particularly feasible at this point.

  In the final section of the paper (Section~\ref{sec: classification}), we move from the study of clones in their own right to the study of metalogical properties of the logics they induce. To be more precise, let a set $X$ with a designated subset $F$ be given. Then each clone $\clone{C}$ on $X$ defines a consequence relation between subsets of $\clone{C}$. For $\Gamma(\tuple{x}), \Delta(\tuple{x})\subseteq \clone{C}$, where $\tuple{x}$ denotes any tuple of variables, the consequence relation $\Gamma(\tuple{x}) \vdash \Delta(\tuple{x})$ holds if and only if
\begin{align*}
  \Gamma(\tuple{a}) \subseteq F \implies \Delta(\tuple{a}) \subseteq F \text{ for all tuples } \tuple{a} \in X.
\end{align*}
  This is merely a clone-theoretic formulation of the usual definition of a logic determined by a logical matrix (i.e.\ by an algebra equipped with a subset).

  We may now inquire into the properties of this consequence relation. For example, we call $\clone{C}$ \emph{protoalgebraic} if there is a set of binary functions $\Delta(x, y) \subseteq \clone{C}$ which satisfies Reflexivity and Modus Ponens:
\begin{align*}
  \emptyset & \vdash \Delta(x, x) & & \text{and} & x, \Delta(x, y) \vdash y.
\end{align*}
  Similarly, we call~$\clone{C}$ \emph{selfextensional} if for each pair of functions $f, g \in \clone{C}$ of arity $n$ such that $f(\tuple{x}) \vdash g(\tuple{x})$ and $g(\tuple{x}) \vdash f(\tuple{x})$ we have $f = g$. These are two of the most important classes in the so-called Leibniz and Frege hierarchy of logics~\cite{font16}, respectively. (The definition of selfextensionality in abstract algebraic logic is in fact somewhat more complicated: see Section~\ref{sec: classification}. For the purposes of this introduction, however, this simpler definition will do.)

  Observe that if $\clone{C} \subseteq \clone{D}$ and $\clone{C}$ is protoalgebraic ($\clone{D}$ is selfextensional), then so is $\clone{D}$ (so is $\clone{C}$). In fact, each selfextensional clone is contained in some maximal selfextensional clone. Whether each protoalgebraic clone extends some minimal protoalgebraic clone is less clear in general. However, it will at least be true for the four-valued clones above $\DMAClone$.

  More explicitly, given a set $X$ and a subset $F \subseteq X$ (and a clone~$\clone{C}$), we may now consider the following problems:
\begin{align*}
  \text{describe the minimal $F$-proto\-algebraic clones on $X$ (above or below $\clone{C}$).}
\end{align*}
  Similarly:
\begin{align*}
  \text{describe the maximal $F$-selfextensional clones on $X$ (above or below $\clone{C}$).}
\end{align*}
  To the best of our knowledge, such questions have not been considered in the literature so far. The closest related result appears to be the theorem of Avron \& B\'{e}ziau~\cite{avron+beziau17} that each para\-consistent three-valued logic which enjoys the deduction theorem is non-selfextensional. In~the last part of this paper, we answer these two questions (and other related ones) for De~Morgan clones which extend $\DMAClone$. In particular, in turns out that only two of these clones are both protoalgebraic and selfextensional.

  While the signature-independent perspective has long been adopted in universal algebra, it does not appear to be particularly common in the study of non-classical logics. This paper can be seen as an attempt to go some way towards convincing non-classical logicians of the merits of a clone-theoretic perspective. Universal algebraists need no such convincing, on the other hand, they may not be aware of the existence of the Leibniz and Frege hierarchies, which classify logics -- hence by extension also clones -- according to the behavior of the so-called Leibniz and Frege operators. The Leibniz hierarchy can be seen as an analogue of the Maltsev hierarchy of universal algebra, although we shall not go into the details of this analogy here (see~\cite{jansana+moraschini1}). Let us merely note the parallel between the existence of terms satisfying certain equations (which defines a Maltsev condition) and the existence of terms satisfying certain logical rules (which defines protoalgebraicity).

\section{Preliminaries}
\label{sec: preliminaries}

  The purpose of this section is to settle our terminology and notation. We first review the few clone-theoretic notions which will be needed in this paper, then we restrict our attention to clones on a four-element set $\DMfour$ and define some important functions and relations on this set.

\begin{definition}[Clones]
  A \emph{clone} on a finite set $X$ is a family of functions on $X$ of finite non-zero \mbox{arity} which is closed under composition of functions and contains all projection maps $\pi^{n}_{k}\colon X^{n} \to X$ onto the $k$-th component.
\end{definition}

  Clones on $X$ form a complete lattice ordered by inclusion, where meets are intersections. The~clone generated by a set $F$ of functions of finite non-zero arity on~$X$, i.e.\ the smallest clone on $X$ which contains $F$, will be denoted $\clonegen{F}$. We~use the notation
\begin{align*}
  \clonegen{f_{1}, \dots, f_{n}} & \assign \clonegen{\{ f_{1}, \dots, f_{n} \}}, \\ \clonegen{X, f_{1}, \dots, f_{n}} & \assign \clonegen{X \cup \{ f_{1}, \dots, f_{n} \}}.
\end{align*}
  Note that the above definition of a clone excludes functions of arity zero, i.e.\ constants. These are instead treated as constant functions of arity one.

  Clones can be specified in essentially two different ways. Firstly, we can provide a generating set for the clone. This amounts to identifying it as the clone of all term functions of some algebra. Secondly, each clone on a finite set can be specified as the set of all functions which preserve some set of relations of finite non-zero arity on $X$. We shall not need a precise definition of preserving an arbitrary relation, as we will be content with certain special cases: preserving a subset of $X$, or a unary function, or a partial order. In~the first part of the paper we provide generating sets for some natural clones on a four-element set defined by preserving certain relations.

  Throughout the paper we consider a fixed four-element set
\begin{align*}
  \DMfour & \assign \{ \True, \False, \Neither, \Both \}.
\end{align*}
  These elements are named according to their logical interpretations: True, False, Neither (True nor False), and Both (True and False).

\begin{definition}[Boolean clones and De~Morgan clones]
  Clones on $\DMfour$ will be called \emph{De~Morgan clones}, while clones on $\Btwo$ will be called \emph{Boolean clones}. A \emph{De~Morgan function} is a function (of finite non-zero arity) on $\DMfour$, while a \emph{Boolean function} is a function (of finite non-zero arity) on~$\Btwo$.
\end{definition}
  
  The most important subsets of $\DMfour$ will be denoted
\begin{align*}
  \Btwo & \assign \{ \True, \False \}, & \Kthree & \assign \{ \True, \Neither, \False \}, & \Pthree & \assign \{ \True, \Both, \False \}.
\end{align*}
  This notation stands for Boole, Kleene, Priest, and De~Morgan. The elements of $\Btwo$ will be called \emph{Boolean elements}. The set $\Kthree$ forms the set of truth values of Kleene's strong three-valued logic \cite{kleene38,kleene52}, while $\Pthree$ forms the set of truth values of the so-called Logic of Paradox of Priest~\cite{priest79}. The~whole set $\DMfour$ is the set of truth values of the four-valued Belnap--Dunn logic, also known as the logic of first-degree entailment~\cite{dunn76,belnap77a,belnap77b}. See~\cite{priest08intro} for more information.

  There are two natural partial orders on the set $\DMfour$: the \emph{truth order} $x \leq y$ and the \emph{information order} $x \infleq y$. The smallest element in the truth order is $\False$, the largest element is $\True$, and the elements $\Neither$ and $\Both$ are incomparable. The smallest element in the information order is $\Neither$, the largest element is $\Both$, and the elements $\True$~and~$\False$ are incomparable. That is, the truth order is the usual bottom--up order and the information order is the left--right order in the following picture:
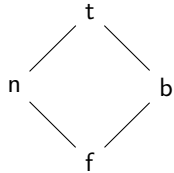
\begin{center}
\begin{tikzpicture}[scale=1,
  dot/.style={circle,fill,inner sep=2.5pt,outer sep=2.5pt}]
  \node (DM4a) at (0,-1) {$\False$};
  \node (DM4b) at (-1,0) {$\Neither$};
  \node (DM4c) at (1,0) {$\Both$};
  \node (DM4d) at (0,1) {$\True$};
  \draw[-] (DM4a) edge (DM4b);
  \draw[-] (DM4a) edge (DM4c);
  \draw[-] (DM4b) edge (DM4d);
  \draw[-] (DM4c) edge (DM4d);
\end{tikzpicture}
\end{center}

  Joins and meets in the truth order are denoted $x \wedge y$ and $x \vee y$, while joins and meets in the information order are denoted $x \infwedge y$ and $x \infvee y$. This yields the \emph{truth lattice} $\langle \DMfour, \wedge, \vee, \True, \False \rangle$ and the \emph{information lattice} $\langle \DMfour, \infwedge, \infvee, \Both, \Neither \rangle$. These two distributive lattices are isomorphic, in fact there are two equally good isomorphisms between them.

  We shall say that an element $a$ of $\DMfour$ is \emph{true} if $a \in \{ \True, \Both \}$ and \emph{false} if $a \in \{ \False, \Both \}$. These sets are the only prime filters on the information lattice, while the sets $\{ \True, \Both \}$ and $\{ \True, \Neither \}$ are the only prime filters on the truth lattice. In~particular, it will be useful to keep in mind that
\begin{align*}
  a \leq b \text{ if and only if } (a \in \{ \True, \Neither \} \implies b \in \{ \True, \Neither \}) \, \& \, (a \in \{ \True, \Both \} \implies b \in \{ \True, \Both \}), \\
  a \infleq b \text{ if and only if } (a \in \{ \True, \Both \} \implies b \in \{ \True, \Both \}) \, \& \, (a \in \{ \False, \Both \} \implies b \in \{ \False, \Both \}).
\end{align*}

  There are also two natural involutions on the set $\DMfour$: the \emph{De~Morgan negation} $\dmneg x$ and the \emph{conflation} $\dual x$. (The notation $\dmneg x$ is meant to suggest reflection across a horizontal axis of symmetry, both in $\DMfouralg$ and in the notation for meets and joins: $\wedge$ transforms into $\vee$. The notation $\dual x$ alludes to the truth value $\dual a$ being dual to $a$ in the sense that it is true if and only if $a$ is non-false, and false if and only if $a$ is non-true. Note that some authors use $\dmneg x$ to denote conflation instead.) De~Morgan negation is uniquely determined by the requirement that $\dmneg a$ is true if and only if $a$ is false, and $\dmneg a$ is false if and only if $a$ is true. Likewise, conflation is uniquely determined by the requirement that $\dual a$ is true if and only if $a$ is non-false, and $\dual a$ is false if and only if $a$ is non-true. In other words,
\begin{align*}
  \dmneg \True & = \False, & \dmneg \Neither & = \Neither, & \dual \True & = \True, & \dual \Neither & = \Both, \\
  \dmneg \False & = \True, & \dmneg \Both & = \Both, & \dual \False & = \False, & \dual \Both & = \Neither.
\end{align*}
  Composing these two involutions yields the Boolean negation.

  These involutions allow us to define the~\emph{De~Morgan dual} of a De~Morgan function $f(\tuple{x})$ as the function $\dmneg f(\dmneg \tuple{x})$ and the \emph{conflation dual} of $f(\tuple{x})$ as the function $\dual f(\dual \tuple{x})$. De~Morgan duality thus transforms meets into joins and vice versa in the truth lattice while commuting with meets and joins in the information lattice. Conflation duality exhibits the opposite behaviour: it transforms meets into joins and vice versa in the information lattice while commuting with meets and joins in the truth lattice.

  There is also what we call the \emph{truth--information symmetry}, which extends the isomorphism between the truth lattice and the information lattice. It~states that the algebras
\begin{align*}
  & \langle \DMfour, \wedge, \vee, \True, \False, \infwedge, \infvee, \Both, \Neither, \dmneg, \dual \rangle & & \text{and} & & \langle \DMfour, \infwedge, \infvee, \Both, \Neither, \vee, \wedge, \True, \False, \dual, \dmneg \rangle\phantom{.}
\end{align*}
  are isomorphic, as are
\begin{align*}
  & \langle \DMfour, \wedge, \vee, \True, \False, \infwedge, \infvee, \Both, \Neither, \dmneg, \dual \rangle & & \text{and} & & \langle \DMfour, \infvee, \infwedge, \Neither, \Both, \wedge, \vee, \True, \False, \dual, \dmneg \rangle.
\end{align*}
  The above symmetries will often be useful to cut our work in half.

  There are also the unary maps $\Box \colon \DMfour \to \Btwo$ and $\Delta \colon \DMfour \to \Btwo$. These are uniquely determined by their codomain and the requirement that $\Box a$ is true if and only if $a$ is exactly true (true and not false), and $\Delta a$ is true if and only if $a$ is true. In other words,
\begin{align*}
  & \Delta a = \True \text{ if } a \in \{ \True, \Both \}, & & \Box a = \True \text{ if } a = \True, \\
  & \Delta a = \False \text{ if } a \in \{ \False, \Neither \}, & & \Box a = \False \text{ if } a < \True.
\end{align*}
  Their De~Morgan duals are the operations $\nabla x$ and $\Diamond x$ such that
\begin{align*}
  & \nabla a = \True \text{ if } a \in \{ \True, \Neither \}, & & \Diamond a = \True \text{ if } a > \False, \\
  & \nabla a = \False \text{ if } a \in \{ \False, \Both \}, & & \Diamond a = \False \text{ if } a = \False.
\end{align*}
  The map $\nabla x$ is the conflation dual of $\Delta x$, while $\Box x$ is its own conflation dual. The operations $\Box x$ and $\Diamond x$ are definable in the following ways:
\begin{align*}
  \Box x & = \Delta x \wedge \nabla x = x \wedge \dual x, & \Diamond x & = \Delta x \vee \nabla x = x \vee \dual x.
\end{align*}

  We will sometimes also need to refer to other De~Morgan functions. In~order to avoid having to introduce and memorize further notation, in such cases we use the convention that $f_{\truthvalue{x}\mapsto\truthvalue{y}}$ denotes the function which behaves like~$f$ except in mapping $\truthvalue{x}$ to $\truthvalue{y}$ rather than to $f(\truthvalue{x})$. For example,
\settowidth{\auxlength}{$\Both$}
\settowidth{\auxlengthtwo}{$\Neither$}
\settowidth{\auxlengththree}{$a$}
\begin{align*}
  & \Truebtob a = \hbox to \auxlength{\hfil$\Both$\hfil} \text{ if } a = \Both, & & \Truenton a = \hbox to \auxlengthtwo{\hfil$\Neither$\hfil} \text{ if } a = \Neither, \\
  & \Truebtob a = \hbox to \auxlength{\hfil$\True$\hfil} \text{ otherwise}, & & \Truenton a = \hbox to \auxlengthtwo{\hfil$\True$\hfil} \text{ otherwise}, \\[5pt]
  & \idbton a = \Neither \text{ if } a = \hbox to \auxlengththree{\hfil$\Both$\hfil}, & & \idntob a = \hbox to \auxlengththree{\hfil$\Both$\hfil} \text{ if } a = \Neither, \\
  & \idntob a = \hbox to \auxlengththree{\hfil$a$\hfil} \text{ otherwise}, & & \idntob a = \hbox to \auxlengththree{\hfil$a$\hfil} \text{ otherwise}.
\end{align*}

\section{\texorpdfstring{De~Morgan clones related to the algebra $\DMfouralg$}{De Morgan clones related to the algebra DM4}}
\label{sec: structure preserving}

  In this section, we consider De~Morgan clones which preserve the following structure naturally associated with the algebra~$\DMfouralg = \langle \DMfour, \wedge, \vee, \True, \False, \dmneg \rangle$:
\begin{itemize}
\item the subalgebras $\Btwoalg$, $\Kthreealg$, $\Pthreealg$,
\item the automorphism $\dual$,
\item the truth order $\leq$, and
\item the information order $\infleq$.
\end{itemize}
  The goal of the present section is to provide (finite) generating sets for all De~Morgan clones defined by conjunctions of these conditions.

  The most important of these clones will be given the following names:
\settowidth{\auxlength}{$\BiLatClone$}
\begin{align*}
  \hbox to \auxlength{\hfil$\DLatClone$\hfil} & \assign \clonegen{\wedge, \vee, \True, \False}, \\
  \hbox to \auxlength{\hfil$\DMAClone$\hfil} & \assign \clonegen{\wedge, \vee, \True, \False, \dmneg}, \\
  \hbox to \auxlength{\hfil$\BiLatClone$\hfil} & \assign \clonegen{\wedge, \vee, \True, \False, \infwedge, \infvee, \Both, \Neither}.
\end{align*}
  These are respectively the clones of term functions of a distributive lattice (given by the truth order), De~Morgan algebra, and bilattice (given by the truth and information orders) over the universe $\{ \True, \False, \Neither, \Both \}$.

  We introduce the following terms for functions which preserve the auto\-morphism $\dual$, the truth order $\leq$, or the information order $\infleq$. Observe that $\dual$ is the only non-trivial auto\-morphism of $\DMfouralg$.

\begin{definition}[Harmonious, positive, and persistent functions]
  A De Morgan function $f(\tuple{x})$ is
\begin{itemize}
\item \emph{harmonious} if $f(\dual \tuple{a}) = \dual f(\tuple{a})$,
\item \emph{positive} if $\tuple{a} \leq \tuple{b}$ implies $f(\tuple{a}) \leq f(\tuple{b})$,
\item \emph{persistent} if $\tuple{a} \infleq \tuple{b}$ implies $f(\tuple{a}) \infleq f(\tuple{b})$.
\end{itemize}
\end{definition}

  The term \emph{harmonious} refers to the fact that for such functions the truth and falsity conditions are in harmony (the falsity conditions coincide with the non-truth conditions), rather like the introduction and the elimination rules in natural deduction calculi are in harmony. The term \emph{persistent} refers to the fact that obtaining more information about atomic propositions yields more information about complex propositions, rather like how the connectives of intuitionistic logic propagate the persistency of atomic propositions to complex propositions in the Kripke semantics for intuitionistic logic.

\begin{definition}[Preserving subalgebras]
  A De~Morgan function $f$ \emph{preserves a set} $X \subseteq \DMfour$ if $\tuple{a} \subseteq X$ implies $f(\tuple{a}) \in X$.
\end{definition}

  The above conditions interact in several ways, thus reducing the number of distinct combinations that we need to consider. In particular, the following observations will be useful throughout this section:
\begin{itemize}
\item Each harmonious function preserves $\Btwo$.
\item Each $\Btwo$-preserving persistent function preserves $\Kthree$ and $\Pthree$.
\item Each function which preserves both $\Kthree$ and $\Pthree$ also preserves $\Btwo$.
\item A harmonious function preserves $\Kthree$ if and only if it preserves $\Pthree$.
\end{itemize}

  It will also be useful to recall that the truth and information orders are related by the following equalities:
\begin{align*}
  x \infwedge y & = ((x \wedge y) \vee \Neither) \wedge ((x \vee y) \vee \Both), & \Both & = \False \infvee \True, \\
  x \infvee y & = ((x \vee y) \wedge \Both) \vee ((x \wedge y) \wedge \Neither), & \Neither & = \False \infwedge \True.
\end{align*}
  It follows that in order to generate $\BiLatClone$ relative to $\DLatClone$ it suffices to add either the constants $\Neither$, $\Both$, or the information join and meet $\infvee$, $\infwedge$. That is,
\begin{align*}
  \BiLatClone & = \clonegen{\DLatClone, \Neither, \Both} = \clonegen{\DLatClone, \infwedge, \infvee}.
\end{align*}

  The first of our theorems now provides generating sets for De~Morgan clones defined by harmonicity plus some conjunction of the above conditions. The only exception is the clone of positive persistent harmonious functions, which will be shown to co\-incide with $\DLatClone$ later.

  The proofs of most of the theorems in this section will follow the same general structure. For example, in order to prove that each harmonious function belongs to $\clonegen{\DLatClone, \dmneg, \dual}$, we
\begin{enumerate}[(i)]
\item identify the smallest harmonious De~Morgan functions $f_{\tuple{a}, \Neither}$, $f_{\tuple{a}, \Both}$, and $f_{\tuple{a}, \True}$ in the truth order such that $f_{\tuple{a}, \Neither}(\tuple{a}) = \Neither$, $f_{\tuple{a}, \Both}(\tuple{a}) = \Both$, and $f_{\tuple{a}, \True}(\tuple{a}) = \True$,
\item express each of these functions in terms of the given generating set, and
\item express $f$ as a suitable join of the functions $f_{\tuple{a}, \Neither}$, $f_{\tuple{a}, \Both}$, and $f_{\tuple{a}, \True}$.
\end{enumerate}
  More explicitly, for each tuple $\tuple{a} \subseteq \DMfour$ we describe the unique harmonious function $f_{\tuple{a}, \True}(\tuple{x})$ such that $f_{\tuple{a}, \True}(\tuple{a}) = \True$ and moreover $f_{\tuple{a}, \True} \leq g(\tuple{x})$ for each harmonious function $g(\tuple{x})$ such that $g(\tuple{a}) = \True$. We also do the same for $f_{\tuple{a}, \Neither}$ and $f_{\tuple{a}, \Both}$. We then express these functions in terms of De~Morgan negation, conflation, and functions in $\DLatClone$. Finally, we observe that
\begin{align*}
  f(\tuple{x}) & = \smashoperator{\bigvee_{f(\tuple{a}) = \Neither}} f_{\tuple{a}, \Neither}(\tuple{x}) ~ \vee ~ \smashoperator{\bigvee_{f(\tuple{a}) = \Both}} f_{\tuple{a}, \Both}(\tuple{x}) ~ \vee ~ \smashoperator{\bigvee_{f(\tuple{a}) = \True}} f_{\tuple{a}, \True}(\tuple{x}).
\end{align*}
  This last step will be the same for all of our proofs, we therefore shall not explicitly spell it out in the proofs below.

\begin{theorem}[Harmonious clones] \label{harmonious clones}
  A De Morgan function lies in
\begin{enumerate}[\rm(i)]
\item $\clonegen{\DLatClone, \dmneg, \dual}$ iff it is harmonious,
\item $\clonegen{\DLatClone, \dual}$ iff it is harmonious and positive,
\item $\clonegen{\DLatClone, \dmneg}$ iff it is harmonious and persistent,
\item $\clonegen{\DLatClone, \dmneg, \Box}$ iff it is harmonious and preserves $\Btwo$, $\Kthree$,~$\Pthree$,
\item $\clonegen{\DLatClone, \Box, \Diamond}$ iff it is harmonious, positive, and preserves $\Btwo$, $\Kthree$, $\Pthree$.
\end{enumerate}
\end{theorem}

\begin{proof}
  The left-to-right implications are straightforward. In the right-to-left direction, recall that each harmonious function preserves $\Btwo$.

  (i) There is a smallest harmonious function $f_{\tuple{a}, \True}$ with $f_{\tuple{a}, \True}(\tuple{a}) = \True$, namely
\begin{align*}
  f_{\tuple{a}, \True}(\tuple{x}) & = \True \text{ if } \tuple{x} \in \{ \tuple{a}, \dual \tuple{a} \}, \\
  f_{\tuple{a}, \True}(\tuple{x}) & = \False \text{ if } \tuple{x} \notin \{ \tuple{a}, \dual \tuple{a} \}.
\end{align*}
  There are also smallest harmonious functions $f_{\tuple{a}, \Neither}$ and $f_{\tuple{a}, \Both}$ such that $f_{\tuple{a}, \Neither}(\tuple{a}) = \Neither$ and $f_{\tuple{a}, \Both}(\tuple{a}) = \Both$, provided that $\tuple{a} \nsubseteq \{ \True, \False \}$, namely
\begin{align*}
  f_{\tuple{a}, \Neither}(\tuple{a}) & = \Neither, & f_{\tuple{a}, \Neither}(\dual \tuple{a}) & = \Both, & f_{\tuple{a}, \Neither}(\tuple{x}) & = \False \text{ if } \tuple{x} \notin \{ \tuple{a}, \dual \tuple{a} \}, \\
  f_{\tuple{a}, \Both}(\tuple{a}) & = \Both, & f_{\tuple{a}, \Both}(\dual \tuple{a}) & = \Neither, & f_{\tuple{a}, \Both}(\tuple{x}) & = \False \text{ if } \tuple{x} \notin \{ \tuple{a}, \dual \tuple{a} \}.
\end{align*}
  It remains to show that the functions $f_{\tuple{a}, \Neither}$, $f_{\tuple{a}, \Both}$, and $f_{\tuple{a}, \True}$ belong to the clone generated by the given set. (Recall the remarks preceding this theorem.) Clearly $f_{\tuple{a}, \Both}(\tuple{x}) = \dual f_{\tuple{a}, \Neither}(\tuple{x})$, therefore it suffices to express $f_{\tuple{a}, \True}$ and $f_{\tuple{a}, \Neither}$.

  Recall that $\Box x = x \wedge \dual x$ and $\Diamond x = x \vee \dual x$. Let $g(\tuple{x})$ be the conjunction of the following functions: 
\begin{align*}
&  \Box x_{i} \text{ for each } a_{i} \text{ such that } a_{i} = \True, \\
&  \Box \dmneg x_{i} \text{ for each } a_{i} \text{ such that } a_{i} = \False, \\
&  \Diamond x_{i} \wedge \Diamond \dmneg x_{i} \text{ for each } a_{i} \text{ such that } a_{i} \in \{ \Neither, \Both \}, \\
&  \Box (x_{i} \vee x_{j}) \text{ for each } a_{i}, a_{j} \text{ such that } \{ a_{i}, a_{j} \} = \{ \Neither, \Both \}, \\
&  \Diamond (x_{i} \wedge x_{j}) \text{ for each } a_{i}, a_{j} \text{ such that } a_{i} = a_{j} = \Neither \text{ or } a_{i} = a_{j} = \Both.
\end{align*}
  We claim that $g(\tuple{x}) = f_{\tuple{a}, \True}(\tuple{x})$. Clearly $g$ is a harmonious function such that $g(\tuple{a}) = \True$, hence also $g(\dual \tuple{a}) = \True$. If $\tuple{x} \notin \{ \tuple{a}, \dual \tuple{a} \}$, then there are $x_{i}$ and $x_{j}$ such that $x_{i} \neq a_{i}$ and $x_{j} \neq \dual a_{j}$. If $a_{i} \in \{ \True, \False \}$ or $a_{j} \in \{ \True, \False \}$, then $g(\tuple{x}) = \False$ thanks to the conjuncts of the forms $\Box x$ and $\Box \dmneg x$. If $a_{i} \notin \{ \True, \False \}$ and $x_{i} \in \{ \True, \False \}$, then $g(\tuple{x}) = \False$ thanks to the conjunct $\Diamond x_{i} \wedge \Diamond \dmneg x_{i}$. Likewise, $g(\tuple{x}) = \False$ if $a_{j} \notin \{ \True, \False \}$ and $x_{j} \in \{ \True, \False \}$. Finally, suppose that $a_{i}, a_{j}, x_{i}, x_{j} \in \{ \Neither, \Both \}$. Then $x_{i} = \dual a_{i}$ and $x_{j} = a_{j}$, since $x_{i} \neq a_{i}$ and $x_{j} \neq \dual a_{j}$. It follows that $g(\tuple{x}) = \False$ thanks to the conjunct $\Box (x_{i} \vee x_{j})$ if $a_{i} \neq a_{j}$ and thanks to $\Diamond (x_{i} \wedge x_{j})$ if $a_{i} = a_{j}$.

  Moreover, $f_{\tuple{a}, \Neither}(\tuple{x}) = f_{\tuple{a}, \True}(\tuple{x}) \wedge x_{i}$ if $a_{i} = \Neither$ and $f_{\tuple{a}, \Neither}(\tuple{x}) = f_{\tuple{a}, \True}(\tuple{x}) \wedge \dual x_{i}$ if $a_{i} = \Both$. But by the assumption that $\tuple{a} \nsubseteq \{ \True, \False \}$ some such $a_{i}$ exists.

  (ii) There is a smallest positive harmonious function $f_{\tuple{a}, \True}$ such that $f_{\tuple{a}, \True}(\tuple{a}) = \True$, namely
\begin{align*}
  f_{\tuple{a}, \True}(\tuple{x}) & = \True \text{ if } \tuple{x} \geq \tuple{a} \text{ or } \tuple{x} \geq \dual \tuple{a}, & f_{\tuple{a}, \True}(\tuple{x}) & = \False \text{ otherwise}.
\end{align*}
  There are also smallest positive harmonious functions $f_{\tuple{a}, \Neither}$ and $f_{\tuple{a}, \Both}$ such that $f_{\tuple{a}, \Neither}(\tuple{a}) = \Neither$ and $f_{\tuple{a}, \Both}(\tuple{a}) = \Both$, provided that $\tuple{a} \nsubseteq \{ \True, \False \}$, namely
\begin{align*}
  f_{\tuple{a}, \Neither}(\tuple{x}) = \True & \text{ if } \tuple{x} \geq \tuple{a} \text{ and } \tuple{x} \geq \dual \tuple{a}, &  f_{\tuple{a}, \Both}(\tuple{x}) = \True & \text{ if } \tuple{x} \geq \tuple{a} \text{ and } \tuple{x} \geq \dual \tuple{a}, \\  
  f_{\tuple{a}, \Neither}(\tuple{x}) = \Neither & \text{ if } \tuple{x} \geq \tuple{a} \text{ and } \tuple{x} \ngeq \dual \tuple{a}, &  f_{\tuple{a}, \Both}(\tuple{x}) = \Both & \text{ if } \tuple{x} \geq \tuple{a} \text{ and } \tuple{x} \ngeq \dual \tuple{a}, \\  
  f_{\tuple{a}, \Neither}(\tuple{x}) = \Both & \text{ if } \tuple{x} \ngeq \tuple{a} \text{ and } \tuple{x} \geq \dual \tuple{a}, &  f_{\tuple{a}, \Both}(\tuple{x}) = \Neither & \text{ if } \tuple{x} \ngeq \tuple{a} \text{ and } \tuple{x} \geq \dual \tuple{a}, \\  
  f_{\tuple{a}, \Neither}(\tuple{x}) = \False & \text{ if } \tuple{x} \ngeq \tuple{a} \text{ and } \tuple{x} \ngeq \dual \tuple{a}, &  f_{\tuple{a}, \Both}(\tuple{x}) = \False & \text{ if } \tuple{x} \ngeq \tuple{a} \text{ and } \tuple{x} \ngeq \dual \tuple{a}.
\end{align*}
  It remains to show that $f_{\tuple{a}, \Neither}$, $f_{\tuple{a}, \Both}$, and $f_{\tuple{a}, \True}$ belong to $\clonegen{\DLatClone, \dual}$. Clearly $f_{\tuple{a}, \Both}(\tuple{x}) = \dual f_{\tuple{a}, \Neither}(\tuple{x})$, therefore it suffices to prove this for $f_{\tuple{a}, \True}$ and $f_{\tuple{a}, \Neither}$.

  Let $g(\tuple{x})$ be the conjunction of the following functions:
\begin{align*}
&  \Box x_{i} \text{ for each } a_{i} \text{ such that } a_{i} = \True, \\
&  \Diamond x_{i} \text{ for each } a_{i} \text{ such that } a_{i} > \False, \\
&  \Box (x_{i} \vee x_{j}) \text{ for each } a_{i}, a_{j} \text{ such that } \{ a_{i}, a_{j} \} = \{ \Neither, \Both \}, \\
&  \Diamond (x_{i} \wedge x_{j}) \text{ for each } a_{i}, a_{j} \text{ such that } a_{i} = a_{j} = \Neither \text{ or } a_{i} = a_{j} = \Both.
\end{align*}
  We claim that $g(\tuple{x}) = f_{\tuple{a}, \True}(\tuple{x})$. Clearly $g$ is a positive harmonious function such that $g(\tuple{a}) = \True$, hence $g(\tuple{x}) = \True$ if $\tuple{x} \geq \tuple{a}$ or $\tuple{x} \geq \dual \tuple{a}$. If $\tuple{x} \ngeq \tuple{a}$ and $\tuple{x} \ngeq \dual \tuple{a}$, then there are $x_{i}$ and $x_{j}$ such that $x_{i} \ngeq a_{i}$ and $x_{j} \ngeq \dual a_{j}$. Clearly $a_{i} > \False$ and $a_{j} > \False$, while $x_{i} < \True$ and $x_{j} < \True$. If $a_{i} = \True$ or $a_{j} = \True$, then $g(\tuple{x}) = \False$ thanks to the conjuncts $\Box x_{i}$ and $\Box x_{j}$. If $x_{i} = \False$ or $x_{j} = \False$, then $g(\tuple{x}) = \False$ thanks to the conjuncts $\Diamond x_{i}$ and $\Diamond x_{j}$. Finally, if $a_{i}, a_{j}, x_{i}, x_{j} \in \{ \Neither, \Both \}$, then $g(\tuple{x}) = \False$ thanks to the conjuncts $\Box (x_{i} \vee x_{j})$ and $\Diamond (x_{i} \wedge x_{j})$ as in (i).

  Moreover, $f_{\tuple{a}, \Neither}(\tuple{x}) = f_{\tuple{a}, \True}(\tuple{x}) \wedge x_{i}$ if $a_{i} = \Neither$, and $f_{\tuple{a}, \Neither}(\tuple{x}) = f_{\tuple{a}, \True}(\tuple{x}) \wedge \dual x_{i}$ if $a_{i} = \Both$. But by the assumption that $\tuple{a} \nsubseteq \{ \True, \False \}$ some such $a_{i}$ exists.

  (iii) Recall that each persistent harmonious function preserves both $\Kthree$ and $\Pthree$. There are smallest persistent harmonious functions $f_{\tuple{a}, \Neither}$ and $f_{\tuple{a}, \Both}$ such that $f_{\tuple{a}, \Neither}(\tuple{a}) = \Neither$ and $f_{\tuple{a}, \Both}(\tuple{a}) = \Both$, provided that $\Neither \in \tuple{a}$ in the former case and~$\Both \in \tuple{a}$ in the latter case, namely
\begin{align*}
  f_{\tuple{a}, \Neither}(\tuple{x}) = \Neither & \text { if } \tuple{x} \infleq \tuple{a}, &  f_{\tuple{a}, \Both}(\tuple{x}) = \Both & \text { if } \tuple{x} \infgeq \tuple{a}, \\
  f_{\tuple{a}, \Neither}(\tuple{x}) = \Both & \text{ if } \tuple{x} \infgeq \dual \tuple{a}, & f_{\tuple{a}, \Both}(\tuple{x}) = \Neither & \text{ if } \tuple{x} \infleq \dual \tuple{a}, \\
  f_{\tuple{a}, \Neither}(\tuple{x}) = \False & \text{ otherwise}, & f_{\tuple{a}, \Both}(\tuple{x}) = \False & \text{ otherwise}.
\end{align*}
  It remains to show that $f_{\tuple{a}, \Neither}$, $f_{\tuple{a}, \Both}$, and $f_{\tuple{a}, \True}$ belong to $\DMAClone$.

  Let us define the function $g$ as
\begin{align*}
  g(\tuple{x}) & = \smashoperator{\bigwedge_{a_{i} \in \{ \True, \Neither \}}} x_{i} ~ \wedge ~ \smashoperator{\bigwedge_{a_{i} \in \{ \False, \Neither \}}} \dmneg x_{i}.
\end{align*}
  We claim that $g(\tuple{x}) = f_{\tuple{a}, \Neither}(\tuple{x})$. Clearly $g$ is a persistent harmonious function such that $g(\tuple{a}) = \Neither$ (using the assumption that $\Neither \in \tuple{a}$), hence $g(\tuple{x}) = \Neither$ if $\tuple{x} \infleq \tuple{a}$ and $g(\tuple{x}) = \Both$ if $\tuple{x} \infgeq \dual \tuple{a}$. If $\tuple{x} \ninfleq \tuple{a}$ and $\tuple{x} \ninfgeq \dual \tuple{a}$, then there are $x_{i}$ and $x_{j}$ such that $x_{i} \ninfleq a_{i}$ and $x_{j} \ninfgeq \dual a_{j}$. It follows that $a_{i} \inflneq \Both$ and $a_{j} \inflneq \Both$, while $x_{i} \infgneq \Neither$ and $x_{j} \inflneq \Both$. If $x_{i} \in \{ \True, \False \}$, then either $a_{i} = \Neither$ or $a_{i} = \dmneg x_{i}$, hence $g(\tuple{x}) = \False$. If~$x_{j} \in \{ \True, \False \}$, then either $a_{j} = \Neither$ or $a_{j} = \dmneg x_{j}$, hence $g(\tuple{x}) = \False$. Otherwise $x_{i} = \Both$ and $x_{j} = \Neither$, hence $g(\tuple{x}) \leq \Both \wedge \Neither = \False$.

  It follows by conflation symmetry that $f_{\tuple{a}, \Both}$ can be expressed as
\begin{align*}
  f_{\tuple{a}, \Both}(\tuple{x}) & = \smashoperator{\bigwedge_{a_{i} \in \{ \True, \Both \}}} x_{i} ~ \wedge ~ \smashoperator{\bigwedge_{a_{i} \in \{ \False, \Both \}}} \dmneg x_{i}.
\end{align*}

  There is moreover a smallest persistent harmonious function $f_{\tuple{a}, \True}$ such that $f_{\tuple{a}, \True}(\tuple{a}) = \True$, namely
\begin{align*}
  f_{\tuple{a}, \True}(\tuple{x}) = \True & \text { if } \tuple{x} \in \{ \tuple{a}, \dual \tuple{a} \}, \\
  f_{\tuple{a}, \True}(\tuple{x}) = \Neither & \text{ if } \tuple{x} \inflneq \tuple{a} \text{ or } \tuple{x} \inflneq \dual \tuple{a}, \\
  f_{\tuple{a}, \True}(\tuple{x}) = \Both & \text{ if } \tuple{a} \inflneq \tuple{x} \text{ or } \tuple{x} \infgneq \dual \tuple{a}, \\
  f_{\tuple{a}, \True}(\tuple{x}) = \False & \text{ otherwise}.
\end{align*}

  Let $g(\tuple{x}) \assign f_{\tuple{a}, \Neither}(\tuple{x}) \vee f_{\tuple{a}, \Both}(\tuple{x})$. We claim that $g(\tuple{x}) = f_{\tuple{a}, \True}(\tuple{x})$. Clearly $g$ is a harmonious function such that $g(\tuple{a}) = \True$, since $f_{\tuple{a}, \Neither}(\tuple{a}) \geq \Neither$ and $f_{\tuple{a}, \Both}(\tuple{a}) \geq \Both$, therefore also $g(\dual \tuple{a}) = \True$. If $\tuple{x}$ and $\tuple{a}$ are incomparable in the information order and so are $\tuple{x}$ and $\dual \tuple{a}$, then $f_{\tuple{a}, \Neither}(\tuple{x}) = \False$ and $f_{\tuple{a}, \Both}(\tuple{x}) = \False$, hence $g(\tuple{x}) = \True$.

  If $\tuple{x} \inflneq \tuple{a}$, then $f_{\tuple{a}, \Neither}(\tuple{x}) = \Neither$ and $f_{\tuple{a}, \Both}(\tuple{x}) \leq \Neither$, hence $g(\tuple{x}) = \Neither$. If $\tuple{x} \infgneq \tuple{a}$, then $f_{\tuple{a}, \Neither}(\tuple{x}) \leq \Both$ and $f_{\tuple{a}, \Both}(\tuple{x}) = \Both$, hence $g(\tuple{x}) = \Both$. Likewise, if $\tuple{x} \inflneq \dual \tuple{a}$, then $f_{\tuple{a}, \Neither}(\tuple{x}) \leq \Neither$ and $f_{\tuple{a}, \Both}(\tuple{x}) = \Neither$, hence $g(\tuple{x}) = \Neither$. If $x \infgneq \dual \tuple{a}$, then $f_{\tuple{a}, \Neither}(\tuple{x}) = \Both$ and $f_{\tuple{a}, \Both}(\tuple{x}) \leq \Both$, hence $g(\tuple{x}) = \Both$.

  (iv) Recall that a harmonious function preserves $\Btwo, \Kthree, \Pthree$ if and only if it preserves $\Kthree$, or equivalently if and only if it preserves $\Pthree$. The smallest harmonious function $f_{\tuple{a}, \True}$ such that $f_{\tuple{a}, \True}(\tuple{a}) = \True$, described in (i), preserves $\Btwo$, $\Kthree$, $\Pthree$ and belongs to $\clonegen{\DMAClone, \Box}$. The smallest harmonious function $f_{\tuple{a}, \Neither}$ such that $f_{\tuple{a}, \Neither}(\tuple{a}) = \Neither$, described in (i), preserves $\Pthree$ and belongs to $\clonegen{\DMAClone, \Box}$ whenever $\Neither \in \tuple{a}$. The smallest harmonious function $f_{\tuple{a}, \Both}$ such that $f_{\tuple{a}, \Both}(\tuple{a}) = \Both$, described in (i), preserves $\Kthree$ and belongs to $\clonegen{\DMAClone, \Box}$ whenever $\Both \in \tuple{a}$. The claim now follows.

  (v) The smallest positive harmonious function $f_{\tuple{a}, \True}$ such that $f_{\tuple{a}, \True}(\tuple{a}) = \True$, described in (ii), preserves $\Btwo$, $\Kthree$, $\Pthree$ and belongs to $\clonegen{\DLatClone, \Box, \Diamond}$. There is a smallest positive harmonious function $f_{\tuple{a}, \Neither}$ preserving $\Btwo, \Kthree, \Pthree$ such that $f_{\tuple{a}, \Neither}(\tuple{a}) = \Neither$, provided that $\Neither \in \tuple{a}$, namely
\begin{align*}
  f_{\tuple{a}, \Neither}(\tuple{x}) = \False & \text{ if } \tuple{x} \ngeq \tuple{a} \text{ and } \tuple{x} \ngeq \dual \tuple{a}, \\
  f_{\tuple{a}, \Neither}(\tuple{x}) = \Neither & \text{ if } \tuple{x} \geq \tuple{a} \text{ and } \tuple{x} \ngeq \dual \tuple{a} \text{ and } x_{i} = \Neither \text{ for some } a_{i} = \Neither, \\
  f_{\tuple{a}, \Neither}(\tuple{x}) = \Both & \text{ if } \tuple{x} \geq \dual \tuple{a} \text{ and } \tuple{x} \ngeq \tuple{a} \text{ and } x_{i} = \Both \text{ for some } a_{i} = \Neither, \\
  f_{\tuple{a}, \Neither}(\tuple{x}) = \True & \text{ otherwise}.
\end{align*}
  That is, $f_{\tuple{a}, \Neither}(\tuple{x}) = \True$ if one of the following conditions holds:
\begin{align*}
  \text{either $\tuple{x} \geq \tuple{a} ~ \& ~ \tuple{x} \geq \dual \tuple{a}$ or $\tuple{x} \geq \tuple{a} ~ \& ~ \Neither \notin \tuple{x}$ or $\tuple{x} \geq \dual \tuple{a} ~ \& ~ \Both \notin \tuple{x}$}.
\end{align*}
  We now prove that $f_{\tuple{a}, \Neither}$ belongs to $\clonegen{\DLatClone, \Box, \Diamond}$.

  Let $g$ be the conjunction of the following functions:
\begin{align*}
& x_{i} \text{ for each } a_{i} \text{ such that } a_{i} = \Neither, \\
&  \Box x_{i} \text{ for each } a_{i} \text{ such that } a_{i} = \True, \\
&  \Diamond x_{i} \text{ for each } a_{i} \text{ such that } a_{i} > \False, \\
&  \Box (x_{i} \vee x_{j}) \text{ for each } a_{i}, a_{j} \text{ such that } \{ a_{i}, a_{j} \} = \{ \Neither, \Both \}, \\
& \Diamond (x_{i} \wedge x_{j}) \text{ for each } a_{i}, a_{j} \text{ such that } a_{i} = a_{j} = \Neither \text{ or } a_{i} = a_{j} = \Both.
\end{align*}
  We claim that $g(\tuple{x}) = f_{\tuple{a}, \Neither}(\tuple{x})$. Clearly $g$ is a positive harmonious function preserving $\Btwo, \Kthree, \Pthree$ such that $g(\tuple{a}) = \Neither$ (using the assumption that $\Neither \in \tuple{a}$), hence $g(\tuple{x}) \geq \Neither$ if $\tuple{x} \geq \tuple{a}$ and $g(\tuple{x}) \geq \Both$ if $\tuple{x} \geq \dual \tuple{a}$. It also follows that $g(\tuple{x}) = \True$ if $f_{\tuple{a}, \Neither}(\tuple{x}) = \True$. If $\tuple{x} \geq \tuple{x}$ and $\tuple{x} \ngeq \dual \tuple{a}$ and $x_{i} = \Neither$ for some $a_{i} = \Neither$, then $g(\tuple{x}) = \Neither$ thanks to the conjunct $x_{i}$. It follows by harmonicity that $g(\tuple{x}) = \Both$ if $\tuple{x} \geq \dual \tuple{a}$ and $\tuple{x} \ngeq \tuple{a}$ and $x_{i} = \Both$ for some $a_{i} = \Neither$. Finally, suppose that $\tuple{x} \ngeq \tuple{a}$ and $\tuple{x} \ngeq \dual \tuple{a}$. Then there are $x_{i}$ and $x_{j}$ such that $x_{i} \ngeq a_{i}$ and $x_{j} \ngeq a_{j}$. It now follows that $g(\tuple{x} = \False$ by the same argument as in (ii).

  By conflation symmetry there is also a smallest positive harmonious function $f_{\tuple{a}, \Both}$ preserving $\Btwo$, $\Kthree$, $\Pthree$ such that $f_{\tuple{a}, \Both}(\tuple{a}) = \Both$, and it belongs to the clone $\clonegen{\DLatClone, \Box, \Diamond}$.
\end{proof}

  We now turn to positive persistent clones, possibly preserving $\Btwo$, $\Kthree$, or~$\Pthree$. Our strategy here will be different: instead of describing the minimal such functions, we shall rely on the observation that each positive persistent function is uniquely determined by its values on $\Btwo$.

  For the purposes of the following lemma, a function $f\colon \Btwo^{n} \to \DMfour$ will be called \emph{positive} if $\tuple{a} \leq \tuple{b}$ implies $f(\tuple{a}) \leq f(\tuple{b})$.

\begin{lemma}[Extension Lemma] \label{extension lemma}
\theoremstartswithitemize
\begin{enumerate}[\rm(i)]
\item Each positive $f \colon \Btwo^{n} \to \DMfour$ extends to a function in $\BiLatClone$.
\item Each positive $f \colon \Btwo^{n} \to \Kthree$ extends to a function in $\clonegen{\DLatClone, \Neither}$.
\item Each positive $f \colon \Btwo^{n} \to \Pthree$ extends to a function in $\clonegen{\DLatClone, \Both}$.
\item Each positive $f \colon \Btwo^{n} \to \Btwo$ extends to a function in $\DLatClone$.
\end{enumerate}
\end{lemma}

\begin{proof}
  Consider the functions
\begin{align*}
  g_{\tuple{a}, \True}(\tuple{x}) & \assign \smashoperator{\bigwedge_{a_{i} = \True}} x_{i}.
\end{align*}
  For each \emph{Boolean} tuple $\tuple{a} \subseteq \{ \True, \False \}$ these functions satisfy
\begin{align*}
  g_{\tuple{a}, \True}(\tuple{x}) = \True & \text{ if } \tuple{x} \geq \tuple{a}, \\
  g_{\tuple{a}, \True}(\tuple{x}) = \False & \text{ otherwise}.
\end{align*}
  Let $g_{\tuple{a}, \Neither}(\tuple{x}) \assign g_{\tuple{a}, \True}(\tuple{x}) \wedge \Neither$ and $g_{\tuple{a}, \Both}(\tuple{x}) \assign g_{\tuple{a}, \True}(\tuple{x}) \wedge \Both$. The desired function $g$ can now be expressed as
\begin{align*}
  g(\tuple{x}) & = \smashoperator{\bigvee_{f(\tuple{a}) = \True}} g_{\tuple{a}, \True}(\tuple{x}) ~ \vee ~ \smashoperator{\bigvee_{f(\tuple{a}) = \Neither}} g_{\tuple{a}, \Neither}(\tuple{x}) ~ \vee ~ \smashoperator{\bigvee_{f(\tuple{a}) = \Both}} g_{\tuple{a}, \Both}(\tuple{x}).
\end{align*}
  If $\Neither$ or $\Both$ or both lie outside the range of $f$, then the function $g$ lies in the appropriate clone.
\end{proof}

\begin{lemma}[Uniqueness Lemma] \label{uniqueness lemma}
  Every positive persistent De Morgan function of arity $n$ is uniquely determined by the values it takes on $\Btwo^{n}$.
\end{lemma}

\begin{proof}
  We show that for each positive persistent De~Morgan function $f$ the values $f(\tuple{a}, \Neither)$ and $f(\tuple{a}, \Both)$ are uniquely determined by the values $f(\tuple{a}, \False)$ and $f(\tuple{a}, \True)$. The value $f(\tuple{a})$ can therefore be computed from the values $f(\tuple{x})$ for $\tuple{x} \subseteq \Btwo$ by induction over the number of (non-Boolean) arguments. 

  Let $g(x) = f(\tuple{a}, x)$. We compute the values $g(\Neither)$ and $g(\Both)$ from the values $g(\True)$ and $g(\False)$. If $g(\False) = g(\True)$, then clearly $g(\False) = g(\Neither) = g(\Both) = g(\True)$. Also, if $g(\False) = \True$, then $g(\True) = g(\Neither) = g(\Both) = \True$.

  Suppose first that $g(\False) = \Neither$. Then $g(\Neither) = \Neither$, since $g(\Neither) \infleq g(\False) = \Neither$. Moreover, $g(\True) \in \{ \True, \Neither \}$, since $g(\True) \geq g(\False) = \Neither$. If $g(\True) = \True$, then $g(\Both) = \True$, since $g(\Both) \infgeq g(\True) = \True$ and $g(\Both) \leq g(\False) = \Neither$. If $g(\True) = \Neither$, then $g(\Both) = \Neither$, since $g(\False) = g(\True)$.

  Suppose that $g(\False) = \Both$. Then $g(\Both) = \Both$, since $g(\Both) \infgeq g(\False) = \Both$. Moreover, $g(\True) \in \{ \True, \Both \}$, since $g(\True) \geq g(\False) = \Both$. If $g(\True) = \True$, then $g(\Neither) = \True$, since $g(\Neither) \infleq g(\False) = \True$ and $g(\Neither) \geq g(\False) = \Both$. If $g(\True) = \Both$, then $g(\Neither) = \Both$, since $g(\False) = g(\True)$.

  Finally, suppose that $g(\False) = \False$. If $g(\True) = \True$, then $g(\Neither) = \Neither$ and $g(\Both) = \Both$, since $g(\Neither) \infleq g(\False) = \False$ and $g(\Neither) \infleq g(\True) = \True$ and $g(\Both) \infgeq g(\True) = \True$ and $g(\Both) \infgeq g(\False) = \False$. If $g(\True) = \Neither$, then $g(\Neither) = \Neither$, since $g(\Neither) \infleq g(\True) = \Neither$, and $g(\Both) = \False$, since $g(\Both) \leq g(\True) = \Neither$ and $g(\Both) \infgeq g(\False) = \False$. Likewise, if $g(\True) = \Both$, then $g(\Both) = \Both$ and $g(\Neither) = \False$.
\end{proof}

\begin{theorem}[Positive persistent clones] \label{positive persistent clones}
  A De Morgan function lies in
\begin{enumerate}[\rm(i)]
\item $\BiLatClone$ iff it is positive and persistent,
\item $\DLatClone$ iff it is positive, persistent, and preserves $\Btwo$,
\item $\clonegen{\DLatClone, \Neither}$ iff it is positive, persistent, and preserves $\Kthree$,
\item $\clonegen{\DLatClone, \Both}$ iff it is positive, persistent, and preserves $\Pthree$.
\end{enumerate}
\end{theorem}

\begin{proof}
  The proofs of these claims are entirely analogous, let us therefore only consider item~(iv). Each function of $\clonegen{\DLatClone, \Both}$ is positive, persistent, and preserves $\Pthree$. Conversely, let $f$ be a positive persistent function preserving $\Pthree$ and let $g\colon \Btwo \to \DMfour$ be the function obtained from $f$ by restricting the domain to $\Btwo$. Because $f$ preserves $\Pthree$, $g$ is in fact a function $g\colon \Btwo \to \Pthree$. Then by the Extension Lemma (Lemma~\ref{extension lemma}) there is a De~Morgan function $h$ in $\clonegen{\DLatClone, \Both}$ which extends~$g$. In particular, $h$ is also positive and persistent, therefore $f$ and $h$ coincide by the Uniqueness Lemma (Lemma~\ref{uniqueness lemma}).
\end{proof}

\begin{corollary}[The positive persistent harmonious clone]
  $\DLatClone$ is the clone of all positive persistent harmonious functions.
\end{corollary}

\begin{proof}
  Each harmonious function preserves $\Btwo$, and conversely each function in $\DLatClone$ is harmonious.
\end{proof}

  This is the De~Morgan analogue of the classical result for Boolean clones which states the Boolean clone $\DLatClone$ (the restriction of the De~Morgan clone $\DLatClone$ to $\Btwo$) is the clone of all positive Boolean functions.

  The Extension and Uniqueness Lemmas also immediately imply the following corollary relating the lattices of Boolean and De~Morgan clones below the clone generated by $\{\wedge,\vee,\True,\False\}$ on $\Btwo$ and $\DMfour$.

\begin{corollary}
  The lattice of De~Morgan clones below the De~Morgan clone $\DLatClone$ is isomorphic to the lattice of Boolean clones below the Boolean clone $\DLatClone$, the isomorphism being the restriction to~$\Btwo$.
\end{corollary}

  The structure of the lattice Boolean clones is known, therefore we can describe the lattice of De~Morgan clones below $\DLatClone$ explicitly. In particular, it is countably infinite. An interesting question, which we shall not answer here, is whether the lattice of De~Morgan clones below $\DMAClone$ is also countable.

  The following sequence of lemmas will now enable us to identify every De~Morgan function with a pair of harmonious functions, and thereby to exploit our description of the clone of all harmonious positive functions and the clone of all harmonious persistent functions in order to find generating functions for the clones of all positive and all persistent functions.

\begin{definition}[Truth and falsity conditions]
  The \emph{truth conditions} of a De Morgan function $f$ are defined as the set $f^{-1} \{ \True, \Both \}$, while the \emph{falsity conditions} of $f$ are defined as the set $f^{-1} \{ \False, \Both \}$.
\end{definition}

\begin{lemma}[Truth and Falsity Lemma] \label{truth and falsity lemma}
  For each $T, F \subseteq \DMfour^{n}$ there is a unique $n$-ary De Morgan function with truth conditions $T$ and falsity conditions $F$.
\end{lemma}

\begin{proof}
  It is easy to see that the unique De~Morgan function with truth conditions $T$ and falsity conditions $F$ is the function
\begin{align*}
  f(\tuple{x}) & = \True \text{ if } \tuple{x} \in T \text{ and } \tuple{x} \notin F, \\
  f(\tuple{x}) & = \False \text{ if } \tuple{x} \notin T \text{ and } \tuple{x} \in F, \\
  f(\tuple{x}) & = \Neither \text{ if } \tuple{x} \notin T \text{ and } \tuple{x} \notin F, \\
  f(\tuple{x}) & = \Both \text{ if } \tuple{x} \in T \text{ and } \tuple{x} \in F. \qedhere
\end{align*}
\end{proof}

  We can in fact construct such a function explicitly if we have a function with truth conditions $T$ and a function with falsity conditions $F$.

\begin{lemma}[Combination Lemma] \label{combination lemma}
  Let $f$ and $g$ be a pair of De Morgan functions of the same arity. Then the unique $n$-ary De~Morgan function with the truth conditions of $f$ and the falsity conditions of~$g$ is the function
\begin{align*}
  (\True \infwedge f(x)) \infvee (\False \infwedge g(x)) = (\Neither \vee f(x)) \wedge (\Both \vee g(x)).
\end{align*}
\end{lemma}

\begin{proof}
  We omit the easy proof of this claim (see~\cite[Lemma~2]{arieli+avron17}).
\end{proof}

\begin{lemma}[Harmonization Lemma] \label{harmonization lemma}
  For each (positive) De Morgan function there is a unique (positive) harmonious De Morgan function with the same truth conditions. There is also a unique (positive) harmonious De Morgan function with the same falsity conditions.
\end{lemma}

\begin{proof}
  Let $f$ be a De~Morgan function. There is at most one harmonious De~Morgan function with the same truth (falsity) conditions as $f$, since the falsity (truth) conditions of a harmonious functions are uniquely determined by its truth (falsity) conditions, and a De~Morgan function is uniquely determined by its truth and falsity conditions.

  To prove that there is a harmonious De~Morgan function with the same truth conditions as $f$, consider the De~Morgan function $g$ such that
\begin{align*}
  g(\tuple{a}) = \True & \text{ if } f(\tuple{a}) \geq \True \text{ and } f(\dual \tuple{a}) \geq \True, \\
  g(\tuple{a}) = \Both & \text{ if } f(\tuple{a}) \geq \True \text{ and } f(\dual \tuple{a}) \ngeq \True, \\
  g(\tuple{a}) = \Neither & \text{ if } f(\tuple{a}) \ngeq \True \text{ and } f(\dual \tuple{a}) \geq \True, \\
  g(\tuple{a}) = \False & \text{ if } f(\tuple{a}) \ngeq \True \text{ and } f(\dual \tuple{a}) \ngeq \True.
\end{align*}
  This function is harmonious by definition and it has the same truth conditions as $f$. Moreover, if $f$ is positive and $\tuple{a} \leq \tuple{b}$, then
\begin{align*}
  g(\tuple{a}) \in \{ \True, \Both \} \implies f(\dual \tuple{a}) \geq \True \implies f(\dual \tuple{b}) \geq \True \implies g(\tuple{b}) \in \{ \True, \Both \}, \\
  g(\tuple{a}) \in \{ \True, \Neither \} \implies f(\dual \tuple{a}) \geq \True \implies f(\dual \tuple{b}) \geq \True \implies g(\tuple{b}) \in \{ \True, \Neither \},
\end{align*}
  therefore $g$ is also positive.

  To prove that there is a harmonious De~Morgan function with the same falsity conditions as $f$, consider the function $h(\tuple{x}) \assign \dmneg f(\dmneg \tuple{x})$. Then there is a harmonious function $g$ with the same truth conditions as $h$, and $i(\tuple{x}) \assign \dmneg g(\dmneg \tuple{x})$ is a De~Morgan function with the same falsity conditions as~$f$. Moreover, if $f$ is positive, then so is $h$, and therefore also $g$ and $i$.
\end{proof}

\begin{theorem}[The positive clone and the persistent clone] \label{positive clone and persistent clone}
\theoremstartswithitemize
\begin{enumerate}[\rm(i)]
\item $\clonegen{\BiLatClone, \dmneg, \dual}$ is the clone of all De~Morgan functions.
\item $\clonegen{\BiLatClone, \dual}$ is the clone of all positive De Morgan functions.
\item $\clonegen{\BiLatClone, \dmneg}$ is the clone of all persistent De Morgan functions.
\end{enumerate}
\end{theorem}

\begin{proof}
  The left-to-right inclusions are again straightforward.

  (i, ii) Consider a De~Morgan function~$f$. By the Combination Lemma (Lemma~\ref{combination lemma}) it suffices to find a De~Morgan function ${g \in \clonegen{\DMAClone, \dual}}$ with the same truth conditions as $f$ and a De~Morgan function ${h \in \clonegen{\DMAClone, \dual}}$ wih the same falsity conditions as $f$. Thanks to the Harmonization lemma (Lemma~\ref{harmonization lemma}) there are harmonious functions $g$ and $h$ with these properties, and by Theorem~\ref{harmonious clones} each harmonious function belongs to $\clonegen{\DMAClone, \dual}$. If $f$ is positive, then by the Harmonization Lemma we can take $g$ and $h$ to be positive, therefore by Theorem~\ref{harmonious clones} they belong to $\clonegen{\DLatClone, \dual}$.

  (iii) This claim follows from (ii) by truth--information symmetry.
\end{proof}

  The theorems proved so far in particular show that the De Morgan clones defined by some conjunction of harmonicity, persistence, and positivity form precisely the lattice shown in Figure~\ref{fig:lattice clones}.

\begin{figure}
\begin{center}
\begin{tikzpicture}[scale=2.4]
  \node (dlat) at (0,0) {$\DLatClone$};
  \node (dual) at (-1,1) {$\clonegen{\DLatClone, \dual}$};
  \node (dmneg) at (1,1) {$\clonegen{\DLatClone, \dmneg}$};
  \node (bilat) at (0,1) {$\BiLatClone$};
  \node (bilatdual) at (-1,2) {$\clonegen{\BiLatClone, \dual}$};
  \node (bilatdmneg) at (1,2) {$\clonegen{\BiLatClone, \dmneg}$};
  \node (dualdmneg) at (0,2) {$\clonegen{\BiLatClone, \dual, \dmneg}$};
  \node (top) at (0,3) {$\clonegen{\BiLatClone, \dual, \dmneg}$};
  \draw[-] (dlat) -- (dual) -- (bilatdual) -- (top);
  \draw[-] (dlat) -- (dmneg) -- (bilatdmneg) -- (top);
  \draw[-] (dlat) -- (bilat) -- (bilatdual);
  \draw[-] (dlat) -- (bilat) -- (bilatdmneg);
  \draw[-] (dual) -- (dualdmneg) -- (top);
  \draw[-] (dmneg) -- (dualdmneg);
\end{tikzpicture}
\end{center}
\caption{De Morgan clones defined by some conjunction of harmonicity, persistence, and positivity}
\label{fig:lattice clones}
\end{figure}
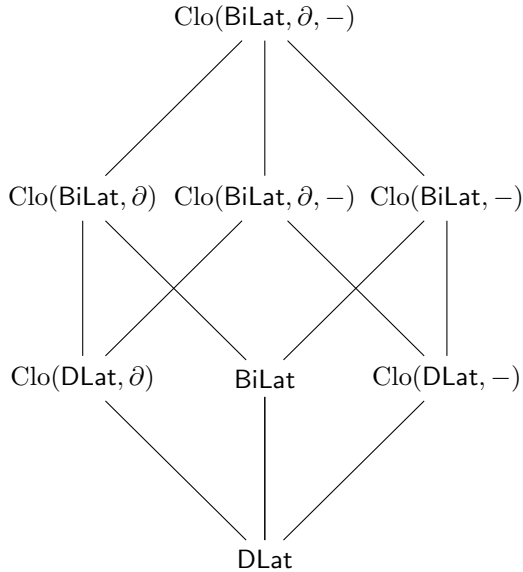

  We now turn to (possibly positive or persistent) clones which preserve $\Btwo$, $\Kthree$, or~$\Pthree$. Here we again use the strategy of describing the minimal functions in the appropriate clones such that $f(\tuple{a})$ is some prescribed value.

\begin{theorem}[Positive clones] \label{positive clones} \label{thm:positive clones}
  A positive De Morgan function lies in
\begin{enumerate}[\rm(i)]
\item $\clonegen{\DLatClone, \Delta, \dual}$ iff it preserves $\Btwo$,
\item $\clonegen{\DLatClone, \Delta, \nabla, \Neither}$ iff it preserves $\Kthree$,
\item $\clonegen{\DLatClone, \Delta, \nabla, \Both}$ iff it preserves $\Pthree$,
\item $\clonegen{\DLatClone, \Delta, \idbton}$ iff it preserves $\Btwo$ and $\Kthree$,
\item $\clonegen{\DLatClone, \Delta, \idntob}$ iff it preserves $\Btwo$ and $\Pthree$,
\item $\clonegen{\DLatClone, \Delta, \nabla}$ iff it preserves $\Btwo$, $\Kthree$, and $\Pthree$.
\end{enumerate}
\end{theorem}

\begin{proof}
  (i) There is a smallest positive function $f_{\tuple{a}, \True}$ with $f_{\tuple{a}, \True}(\tuple{a}) = \True$, namely
\begin{align*}
  f_{\tuple{a}, \True}(\tuple{x}) = \True & \text{ if } \tuple{a} \leq \tuple{x}, \\
  f_{\tuple{a}, \True}(\tuple{x}) = \False & \text{ otherwise}.
\end{align*}
  Moreover, this function preserves $\Btwo$, $\Kthree$, $\Pthree$. Recall that $\nabla x = \Delta \dual x$. Let us define the function $g$ as
\begin{align*}
  g(\tuple{x}) = \smashoperator{\bigwedge_{a_{i} \in \{ \True, \Both \}}} \Delta x_{i} ~ \wedge ~ \smashoperator{\bigwedge_{a_{i} \in \{ \True, \Neither \}}} \nabla x_{i}.
\end{align*}
  Clearly $g$ is a positive function such that $g(\tuple{a}) = \True$, hence $g(\tuple{x}) = \True$ if $\tuple{x} \geq \tuple{a}$. If $\tuple{x} \ngeq \tuple{a}$, then there is some $x_{i}$ such that $x_{i} \ngeq a_{i}$, therefore either $a_{i} \in \{ \True, \Both \}$ and $x_{i} \notin \{ \True, \Both \}$ or $a_{i} \in \{ \True, \Neither \}$ and $x_{i} \notin \{ \True, \Neither \}$. But then $g(\tuple{x}) = \False$ thanks for the conjuncts $\Delta x_{i}$ and $\Delta \dual x_{i}$. Thus $g(\tuple{x}) = f_{\tuple{a}, \True}(\tuple{x})$.

  There is also a smallest positive $\Btwo$-preserving function $f_{\tuple{a}, \Neither}$ such that $f_{\tuple{a}, \Neither}(\tuple{a}) = \Neither$ and $f_{\tuple{a}, \Both}(\tuple{a}) = \Both$, provided that $\tuple{a} \nsubseteq \{ \True, \False \}$, namely
\begin{align*}
  f_{\tuple{a}, \Neither}(\tuple{x}) = \Neither & \text{ if } \tuple{x} \geq \tuple{a} \text{ and } x_{i} = a_{i} \text{ for some } a_{i} \in \{ \Neither, \Both \}, \\
  f_{\tuple{a}, \Neither}(\tuple{x}) = \True & \text{ if } \tuple{x} \geq \tuple{a} \text{ and } x_{i} = \True, \text{ for each } a_{i} \in \{ \Neither, \Both \}, \\
  f_{\tuple{a}, \Neither}(\tuple{x}) = \False & \text{ otherwise}.
\end{align*}
  Let us define the function $g$ as
\begin{align*}
  g(\tuple{x}) = f_{\tuple{a}, \True}(\tuple{x}) ~ \wedge ~ \smashoperator{\bigwedge_{a_{i} = \Neither}} x_{i} ~ \wedge ~ \smashoperator{\bigwedge_{a_{i} = \Both}} \dual x_{i}.
\end{align*}
  Clearly $g$ is a positive $\Btwo$-preserving function such that $g(\tuple{a}) = \Neither$ (using the assumption that $\tuple{a} \nsubseteq \{ \True, \False \}$), hence $g(\tuple{x}) \geq \Neither$ if $\tuple{x} \geq \tuple{a}$ and $g(\tuple{x}) = \True$ if $\tuple{x} \geq \tuple{a}$ and $x_{i} = a_{i}$ for each $a_{i} \in \{ \Neither, \Both \}$. Moreover, $f_{\tuple{a}, \Neither}(\tuple{x}) \leq f_{\tuple{a}, \True}(\tuple{x}) = \False$ if $\tuple{a} \nleq \tuple{x}$. Finally, suppose that $\tuple{a} \leq \tuple{x}$ and $x_{i} = a_{i}$ for some $a_{i} \in \{ \Neither, \Both \}$. Then $g(\tuple{x}) = \Neither$ thanks to the conjuncts $x_{i}$ and $\dual x_{i}$.

  The smallest positive $\Btwo$-preserving function $f_{\tuple{a}, \Both}(\tuple{x})$ such that $f_{\tuple{a}, \Both}(\tuple{a}) = \Both$ can be expressed similarly by conflation symmetry.

   (ii) The smallest positive $\Kthree$-preserving function $f_{\tuple{a}, \True}$ such that $f_{\tuple{a}, \True}(\tuple{a}) = \True$ is the same as the smallest positive function with this property. The smallest positive $\Kthree$-preserving function $f_{\tuple{a}, \Neither}$ such that $f_{\tuple{a}, \Neither}(\tuple{a}) = \Neither$ is
\begin{align*}
  f_{\tuple{a}, \Neither}(\tuple{x}) = \Neither & \text{ if } \tuple{a} \leq \tuple{x}, \\
  f_{\tuple{a}, \Neither}(\tuple{x}) = \False & \text{ otherwise}.
\end{align*}
  But this is simply the function $f_{\tuple{a}, \True}(\tuple{x}) \wedge \Neither$.

  Provided that $\Both \in \tuple{a}$, the smallest positive $\Kthree$-preserving function $f_{\tuple{a}, \Both}$ such that $f_{\tuple{a}, \Both}(\tuple{a}) = \Both$ is
\begin{align*}
  f_{\tuple{a}, \Both}(\tuple{x}) = \Both & \text { if } \tuple{a} \leq \tuple{x} \text{ and } x_{i} = \Both \text{ for some } a_{i} = \Both, \\
  f_{\tuple{a}, \Both}(\tuple{x}) = \True & \text { if } \tuple{a} \leq \tuple{x} \text{ but } x_{i} = \True \text{ for each } a_{i} = \Both, \\
  f_{\tuple{a}, \Both}(\tuple{x}) = \False & \text{ otherwise}.
\end{align*}
  But this is simply the conjunction of $f_{\tuple{a}, \True}(\tuple{x})$ with each $x_{i}$ such that $a_{i} = \Both$.

  (iv) The smallest positive function $f_{\tuple{a}, \True}$ preserving $\Btwo$ and $\Kthree$ such that $f_{\tuple{a}, \True}(\tuple{a}) = \True$ is the same as the smallest positive function with this property. The smallest positive function $f_{\tuple{a}, \Both}$ preserving $\Btwo$ and $\Kthree$ such that $f_{\tuple{a}, \Both}(\tuple{a}) = \Both$ is the same as the smallest $\Kthree$-preserving function with this property. Finally, the smallest function $f_{\tuple{a}, \Neither}$ preserving $\Btwo$ and $\Kthree$ such that $f_{\tuple{a}, \Neither}(\tuple{a}) = \Neither$ is the same as the smallest $\Btwo$-preserving function with this property. To express this function in terms of $\clonegen{\DLatClone, \Delta, \idbton}$, it suffices to replace the conjuncts $\dual x_{i}$ by $\idbton x_{i}$ in the expression of the smallest $\Btwo$-preserving function with this property.

   (iii, v) These follow from (ii, iv) by conflation symmetry.

  (vi) The smallest positive function $f_{\tuple{a}, \True}$ preserving $\Btwo$, $\Kthree$, $\Pthree$ such that $f_{\tuple{a}, \True}(\tuple{a}) = \True$ is the same as the smallest positive function with this property. The smallest positive function $f_{\tuple{a}, \Both}$ preserving $\Btwo$, $\Kthree$, $\Pthree$ such that $f_{\tuple{a}, \Both}(\tuple{a}) = \Both$ is the same as the smallest $\Kthree$-preserving function with this property, provided that $\Both \in \tuple{a}$. The smallest positive function $f_{\tuple{a}, \Neither}$ preserving $\Btwo$, $\Kthree$, $\Pthree$ such that $f_{\tuple{a}, \Neither}(\tuple{a}) = \Neither$ is the same as the smallest $\Pthree$-preserving function with this property, and by conflation symmetry this is simply the conjunction of $f_{\tuple{a}, \True}(\tuple{x})$ with each $x_{i}$ such that $a_{i} = \Neither$.
\end{proof}

\begin{theorem}[Clones preserving subalgebras] \label{clones preserving subalgebras}
  A De~Morgan function lies in
\begin{enumerate}[\rm(i)]
\item $\clonegen{\DMAClone, \Delta}$ iff it preserves $\Btwo$, $\Kthree$, and $\Pthree$.
\item $\clonegen{\DMAClone, \Delta, \Neither}$ iff it preserves $\Kthree$.
\item $\clonegen{\DMAClone, \Delta, \Both}$ iff it preserves $\Pthree$.
\item $\clonegen{\DMAClone, \Delta, \dual}$ iff it preserves $\Btwo$.
\item $\clonegen{\DMAClone, \Delta, \idbton}$ iff it preserves $\Btwo$ and $\Kthree$.
\item $\clonegen{\DMAClone, \Delta, \idntob}$ iff it preserves $\Btwo$ and $\Pthree$.
\end{enumerate}
\end{theorem}

\begin{proof}
  Let us define the functions $f_{\tuple{a}, b}$ for $b \in \{ \True, \Neither, \Both \}$ as
\settowidth{\auxlength}{$\Both$}
\begin{align*}
  f_{\tuple{a}, \True}(\tuple{x}) = \hbox to \auxlength{\hfil$\True$\hfil} & \text{ if } \tuple{x} = \tuple{a}, & f_{\tuple{a}, \True}(\tuple{x}) = \False & \text{ otherwise}, \\
  f_{\tuple{a}, \Neither}(\tuple{x}) = \Neither & \text{ if } \tuple{x} = \tuple{a}, & f_{\tuple{a}, \Neither}(\tuple{x}) = \False & \text{ otherwise}, \\
  f_{\tuple{a}, \Both}(\tuple{x}) = \Both & \text{ if } \tuple{x} = \tuple{a}, & f_{\tuple{a}, \Both}(\tuple{x}) = \False & \text{ otherwise}.
\end{align*}
  These are the smallest De~Morgan functions $f$ such that $f(\tuple{a}) = \True$, $f(\tuple{a}) = \Neither$, and $f(\tuple{a}) = \Both$, respectively.

  The function $f_{\tuple{a}, \True}$ can be expressed as
\begin{align*}
  f_{\tuple{a}, \True}(\tuple{x}) = \smashoperator{\bigwedge_{a_{i} \in \{ \True, \Both \}}} \Delta x_{i} ~ \wedge ~ \smashoperator{\bigwedge_{a_{i} \in \{ \True, \Neither \}}} \nabla x_{i} ~ \wedge ~ \smashoperator{\bigwedge_{a_{i} \notin \{ \True, \Both \}}} \dmneg \Delta x_{i} ~ \wedge ~ \smashoperator{\bigwedge_{a_{i} \notin \{ \True, \Neither \}}} \dmneg \nabla x_{i}.
\end{align*}
  The functions $f_{\tuple{a}, \Neither}$ and $f_{\tuple{a}, \Both}$ can be expressed as $f_{\tuple{a}, \Neither}(\tuple{x}) = f_{\tuple{a}, \True}(\tuple{x}) \wedge \Neither$ and $f_{\tuple{a}, \Both}(\tuple{x}) = f_{\tuple{a}, \True}(\tuple{x}) \wedge \Both$. If $a_{i} = \Neither$, then $f_{\tuple{a}, \Neither}(\tuple{x}) = f_{\tuple{a}, \True}(\tuple{x}) \wedge x_{i}$. Likewise, if $a_{i} = \Both$, then $f_{\tuple{a}, \Both}(\tuple{x}) = f_{\tuple{a}, \True}(\tuple{x}) \wedge x_{i}$. This proves (i)--(iii).

  If $a_{i} = \Both$, then $f_{\tuple{a}, \Neither}(\tuple{x}) = f_{\tuple{a}, \True}(\tuple{x}) \wedge \dual x_{i} = f_{\tuple{a}, \True}(\tuple{x}) \wedge \idbton x_{i}$. Likewise, if $a_{i} = \Neither$, then $f_{\tuple{a}, \Both}(\tuple{x}) = f_{\tuple{a}, \True}(\tuple{x}) \wedge \dual x_{i} = f_{\tuple{a}, \True}(\tuple{x}) \wedge \idntob x_{i}$. This proves (iv)--(vi).
\end{proof}

  In all of the cases considered above, the clones were always obtained by adding some at most unary functions to the clones $\DLatClone$ or $\BiLatClone$. Moreover, with the possible exception of the clones of all positive De~Morgan functions preserving $\Btwo$ and either $\Kthree$ or $\Pthree$, these functions were very natural ones: $\dmneg, \dual, \Delta, \nabla, \Neither, \Both$. In~the last case that we consider in this section, namely the case of persistent functions preserving $\Btwo$, $\Kthree$, or $\Pthree$, the situation will be somewhat less satisfying.

  In order to describe these clones, we introduce the two persistent binary functions $\pbp_{1}, \pbp_{2}$ (``persistent Boolean-preserving'') shown in Figure~\ref{fig: persistent binary functions}. Observe that the functions $\Truebtob x$ and $\Truenton x$ can be expressed as
\begin{align*}
  \Truebtob x & = \pbp_{1} (x, \True), \\
  \Truenton x & = \pbp_{2}  (x, \True).
\end{align*}

\begin{figure}
\begin{center}
\begin{tabular}{r | c c c c }
  $\pbp_{1}$ & $\True$ & $\False$ & $\Neither$ & $\Both$ \\
  \hline
  $\True$ & $\True$ & $\False$ & $\False$ & $\Both$ \\
  $\False$ & $\True$ & $\False$ & $\False$ & $\Both$ \\
  $\Neither$ & $\True$ & $\False$ & $\Neither$ & $\Both$ \\
  $\Both$ & $\Both$ & $\False$ & $\False$ & $\Both$
\end{tabular}
\qquad
\begin{tabular}{r | c c c c }
  $\pbp_{2}$ & $\True$ & $\False$ & $\Neither$ & $\Both$ \\
  \hline
  $\True$ & $\True$ & $\False$ & $\Neither$ & $\False$ \\
  $\False$ & $\True$ & $\False$ & $\Neither$ & $\False$ \\
  $\Neither$ & $\Neither$ & $\False$ & $\Neither$ & $\False$ \\
  $\Both$ & $\True$ & $\False$ & $\Neither$ & $\Both$
\end{tabular}
\end{center}
\caption{Binary persistent functions preserving $\Btwo$}
\label{fig: persistent binary functions}
\end{figure}

\begin{theorem}[Persistent clones] \label{persistent clones}
  A De Morgan function lies in
\begin{enumerate}[\rm(i)]
\item $\clonegen{\DMAClone, \pbp_{1}, \pbp_{2}}$ iff it is persistent and preserves $\Btwo$,
\item $\clonegen{\DMAClone, \pbp_{1}, \pbp_{2}, \Neither}$ iff it is persistent and preserves $\Kthree$,
\item $\clonegen{\DMAClone, \pbp_{1}, \pbp_{2}, \Both}$ iff it is persistent and preserves $\Pthree$.
\end{enumerate}
\end{theorem}

\begin{proof}
  (i) Recall that a persistent function which preserves $\Btwo$ also preserves $\Kthree$ and $\Pthree$. There is a smallest persistent $\Btwo$-preserving function $f_{\tuple{a}, \Neither}$ such that $f_{\tuple{a}, \Neither}(\tuple{a}) = \Neither$, provided that $\Neither \in \tuple{a}$, namely
\begin{align*}
  f_{\tuple{a}, \Neither}(\tuple{x}) & = \Neither \text { if } \tuple{x} \infleq \tuple{a}, \\
  f_{\tuple{a}, \Neither}(\tuple{x}) & = \False \text{ if } \tuple{x} \ninfleq \tuple{a}.
\end{align*}
  Let us define the function $g$ as
\begin{align*}
  g(\tuple{x}) = \smashoperator{\bigwedge_{a_{i} = \True}} x_{i} ~ \wedge ~ \smashoperator{\bigwedge_{a_{i} = \False}} \dmneg x_{i} ~ \wedge ~ \smashoperator{\bigwedge_{a_{i} = \Neither}} \dmneg \Truenton x_{i}.
\end{align*}
  We claim that $g(\tuple{x}) = f_{\tuple{a}, \Neither}(\tuple{x})$. The function $g$ is a persistent $\Btwo$-preserving function such that $g(\tuple{a}) = \Neither$, therefore $g(\tuple{x}) = \Neither$ if $\tuple{x} \infleq \tuple{a}$. If $\tuple{x} \ninfleq \tuple{a}$, then $x_{i} \ninfleq a_{i}$ for some $x_{i}$. It follows that $a_{i} \inflneq \Both$ and $x_{i} \infgneq \Neither$. If $a_{i} = \Neither$, then $g(\tuple{x}) = \False$ thanks to the conjunct $\dmneg \Truenton x_{i}$. We may therefore assume that $x_{i} = a_{i}$ whenever $a_{i} = \Neither$. If $a_{i} \in \{ \True, \False \}$ and $x_{i} \in \{ \True, \False \}$, then $x_{i} = \dmneg a_{i}$, hence $g(\tuple{x}) = \False$ thanks to the conjunct $x_{i}$ or $\dmneg x_{i}$. If $a_{i} \in \{ \True, \False \}$ and $x_{i} = \Both$ and $x_{i} = a_{i}$ whenever $a_{i} = \Neither$, then $g(\tuple{x}) \leq \Neither \wedge \Both = \False$ because $\Neither \in \tuple{a}$.

  By conflation symmetry, it follows that there is also a smallest persistent $\Btwo$-preserving function $f_{\tuple{a}, \Both}$ such that $f_{\tuple{a}, \Both}(\tuple{a}) = \Both$, and it also belongs to $\clonegen{\DMAClone, \Truenton, \Truebtob}$.

  Finally, there is also a smallest persistent $\Btwo$-preserving function $f_{\tuple{a}, \True}$ such that $f_{\tuple{a}, \True}(\tuple{a}) = \True$, namely
\begin{align*}
  f_{\tuple{a}, \True}(\tuple{x}) & = \False \text{ if } \tuple{x} \ninfleq \tuple{a} \text{ and } \tuple{x} \ninfgeq \tuple{a}, \\
  f_{\tuple{a}, \True}(\tuple{x}) & = \Neither \text{ if } \tuple{x} \inflneq \tuple{a} \text{ and } \Neither \in \tuple{x}, \\ 
  f_{\tuple{a}, \True}(\tuple{x}) & = \Both \text{ if } \tuple{x} \infgneq \tuple{a} \text{ and } \Both \in \tuple{x}, \\
  f_{\tuple{a}, \True}(\tuple{x}) & = \True \text { if either } \tuple{x} = \tuple{a} \text{ or } \tuple{x} \inflneq \tuple{a} ~ \& ~ \Neither \notin \tuple{x} \text{ or } \tuple{x} \infgneq \tuple{a} ~ \& ~ \Both \notin \tuple{x}.
\end{align*}
  Let us define $g$ as the conjunction of the following functions:
\begin{align*}
&  x_{i} \text{ for each } a_{i} = \True, \\
&  \dmneg x_{i} \text{ for each } a_{i} = \False, \\
& \Truebtob x_{i} \text{ for each } a_{i} \neq \Both, \\
& \Truenton x_{i} \text{ for each } a_{i} \neq \Neither, \\
& (\dmneg \Truenton x_{i} \vee \dmneg \Truebtob x_{j}) \text{ for each } a_{i} = \Neither \text{ and } a_{j} = \Both, \\
& \pbp_{1}(x_{i}, x_{j}) \text{ for each } a_{i} = \Neither \text{ and } a_{j} = \True, \\
& \pbp_{1}(x_{i}, \dmneg x_{j}) \text{ for each } a_{i} = \Neither \text{ and } a_{j} = \False, \\
& \pbp_{2}(x_{i}, x_{j}) \text{ for each } a_{i} = \Both \text{ and } a_{j} = \True, \\
& \pbp_{2}(x_{i}, \dmneg x_{j}) \text{ for each } a_{i} = \Both \text{ and } a_{j} = \False.
\end{align*}
  We claim that $g(\tuple{x}) = f_{\tuple{a}, \True}(\tuple{x})$. Clearly $g$ is a persistent $\Btwo$-preserving function such that $g(\tuple{a}) = \True$, hence $g(\tuple{x}) \in \{ \True, \Neither \}$ if $\tuple{x} \infleq \tuple{a}$ and $g(\tuple{x}) \in \{ \True, \Both \}$ if $\tuple{x} \infgeq \tuple{a}$. It follows that $g(\tuple{x}) = \True$ if $\tuple{x} = \tuple{a}$ or $\tuple{x} \inflneq \tuple{a} ~ \& ~ \Neither \notin \tuple{x}$ or $\tuple{x} \infgneq \tuple{a} ~ \& ~ \Both \notin \tuple{x}$.

  Now suppose that $\tuple{x} \inflneq \tuple{a}$ and $\Neither \in \tuple{x}$. (The case where $\tuple{x} \infgneq \tuple{a}$ and $\Both \in \tuple{x}$ is related to this case by conflation symmetry.) Then $x_{i} \inflneq a_{i}$ for some $x_{i}$ and $x_{j} = \Neither$ for some $x_{j}$. If $a_{i} \in \{ \True, \False \}$, then $x_{i} = \Neither$ and $g(\tuple{x}) \leq \Neither$, i.e.\ $g(\tuple{x}) = \Neither$, thanks to the conjuncts $x_{i}$ or $\dmneg x_{i}$. If $a_{i} = \Both$ and $x_{i} = \Neither$, then $g(\tuple{x}) \leq \Neither$ thanks to the conjunct $\Truenton x_{i}$. We can therefore assume that $a_{j} = \Neither$, otherwise we take $x_{i}$ to be $x_{j}$. But if $a_{i} = \Both$ and $a_{j} = x_{j} = \Neither$ and $x_{i} \in \{ \True, \False \}$, then $g(\tuple{x}) \leq \Neither$, i.e.\ $g(\tuple{x}) = \Neither$, thanks to the conjunct $(\dmneg \Truenton x_{j} \vee \dmneg \Truebtob x_{i})$.

  Finally, suppose that $\tuple{x} \ninfleq \tuple{a}$ and $\tuple{x} \ninfgeq \tuple{a}$. Then there are $x_{i}$ and $x_{j}$ such that $x_{i} \ninfleq a_{i}$ and $x_{j} \ninfgeq a_{j}$. It follows that $a_{i} \inflneq \Both$, $a_{j} \infgneq \Neither$, $x_{i} \infgneq \Neither$, and $x_{j} \inflneq \Both$. If $a_{i} \in \{ \True, \False \}$ and $x_{i} \in \{ \True, \False \}$, then $x_{i} = \dmneg a_{i}$ and $g(\tuple{x}) = \False$ thanks to the conjunct $x_{i}$ or $\dmneg x_{i}$. Likewise, if $a_{j} \in \{ \True, \False \}$ and $x_{j} \in \{ \True, \False \}$, then $x_{j} = \dmneg a_{j}$ and $g(\tuple{x}) = \False$ thanks to the conjunct $x_{j}$ or $\dmneg x_{j}$. If $x_{i} = \Both$ and $x_{j} = \Neither$, then $g(\tuple{x}) = \False$ thanks to the conjuncts $\Truebtob x_{i}$ and $\Truenton x_{j}$ (we have $a_{i} \neq \Both$ and $a_{j} \neq \Neither$).

  It remains to deal with cases where either $x_{i} \in \{ \True, \False \}$ and $a_{i} = \Neither$ or $x_{j} \in \{ \True, \False \}$ and $a_{j} = \Both$. Suppose therefore that $x_{i} \in \{ \True, \False \}$ and $a_{i} = \Neither$ (the other case will follow by conflation symmetry). If $a_{j} = \Both$, then $g(\tuple{x}) = \False$ thanks to the conjunct $(\dmneg \Truenton x_{i} \vee \dmneg \Truebtob x_{j})$. If $a_{j} = \True$, then $g(\tuple{x}) = \False$ thanks to the conjunct $\pbp_{1}(x_{i}, x_{j})$, since $x_{j} \ninfgeq a_{j}$. If $a_{j} = \False$, then $g(\tuple{x}) = \False$ thanks to the conjunct $\pbp_{1}(x_{i}, \dmneg x_{j})$, again since $x_{j} \ninfgeq a_{j}$.

  The function $f$ can now be expressed as a disjunction of functions of the form $f_{\tuple{a}, b}$ as usual, proving that $f$ lies in $\clonegen{\DMAClone, \pbp_{1}, \pbp_{2}}$.

  (ii) Suppose that $f$ is a persistent De Morgan function preserving~$\Kthree$. We decompose $f$ as a disjunction of simpler functions. It will suffice to express the following functions for each $\tuple{a} \in \DMfour^{n}$:
\begin{align*}
  f_{\tuple{a}, \Neither}(\tuple{x}) & = \Neither \text { if } \tuple{x} \infleq \tuple{a}, & f_{\tuple{a}, \Neither}(\tuple{x}) & = \False \text{ if } \tuple{x} \ninfleq \tuple{a},
\end{align*}
  and the following functions for each $\tuple{a} \in \DMfour^{n}$ such that $\Both \in \tuple{a}$:
\begin{align*}
  f_{\tuple{a}, \Both}(\tuple{x}) & = \Both \text { if } \tuple{x} \infgeq \tuple{a}, & f_{\tuple{a}, \Both}(\tuple{x}) & = \False \text{ if } \tuple{x} \ninfgeq \tuple{a},
\end{align*}
  and the following functions for each $\tuple{a} \in \DMfour^{n}$:
\begin{align*}
  f_{\tuple{a}, \True}(\tuple{x}) & = \True \text { if } \tuple{x} = \tuple{a} \text{ or } (\tuple{x} \infgneq \tuple{a} \text{ and } \Both \notin \tuple{x}), & f_{\tuple{a}, \True}(\tuple{x}) & = \False \text{ else}.
\end{align*}
  Each persistent De Morgan function which preserves $\Kthree$ is then a disjunction of these functions.

  The functions $f_{\tuple{a}, \Both}$ have already been dealt with in the proof of (i), as has $f_{\tuple{a}, \Neither}$ in case $\Neither \in \tuple{a}$. If $\Neither \notin \tuple{a}$, it suffices to add $\Neither$ as a conjunct to the expression for $f_{\tuple{a}, \Neither}$ in (ii). It thus remains to express the functions $f_{\tuple{a}, \True}$. To do so, it suffices to add the following conjuncts to the function $g$ from (i) to obtain the function~$h$:
\begin{align*}
  & \Neither \vee x_{i} \text{ for each } a_{i} = \Both, \\
  & \Neither \vee \dmneg x_{i} \text{ for each } a_{i} = \Both.
\end{align*}
  The only difference compared to (i) occurs in case $\tuple{x} \inflneq \tuple{a}$ if $a_{i} = \Both$ and $x_{i} \in \{ \True, \False \}$, where the conjuncts $\Neither \vee x_{i}$ and $\Neither \vee \dmneg x_{i}$ ensure that $h(\tuple{x}) = \Neither$.

  (iii) follows from (ii) by conflation symmetry.
\end{proof}

  Let us remark that the use of binary functions cannot be avoided in the last theorem. This is because by a routine but tedious case analysis we can show that each unary persistent function preserving $\Btwo$ is definable over $\DMAClone$ using the functions $\Truebtob$ and $\Truenton$, but the clone $\clonegen{\DMAClone, \Truebtob, \Truenton}$ does not include the functions $\pbp_{1}$ and $\pbp_{2}$ because the subsets of $\DMfour^{2}$
\begin{align*}
  \{ \pair{\True}{\True}, \pair{\Neither}{\True}, \pair{\Neither}{\False}, \pair{\False}{\False} \}, \\
  \{ \pair{\True}{\True}, \pair{\Both}{\True}, \pair{\Both}{\False}, \pair{\False}{\False} \},
\end{align*}
  are closed under all operations of this clone but the first subset is not closed under $\pbp_{1}$ and the second is not closed under $\pbp_{2}$. Because the first set is closed under $\pbp_{2}$ and the second set is closed under $\pbp_{1}$, the same example also shows that the generating sets in the last theorem are irredundant.

\section{De~Morgan clones above \texorpdfstring{$\DMAClone$}{DMA}}
\label{sec: above dma}

  In the previous section we described the De Morgan clones which preserve some of the structure of the algebra $\DMfouralg$. The goal of the current section is to describe the De Morgan clones above $\DMAClone$ which \emph{fail} to preserve this structure. For example, each clone above $\DMAClone$ which is not harmonious must contain one of ten at most binary non-harmonious functions which we list explicitly. This will enable us to describe the upper covers of $\DMAClone$ in the lattice of all De~Morgan clones.

\begin{theorem}[Clones not preserving subalgebras] \label{clones not preserving subalgebras}
  A De~Morgan clone above $\DMAClone$ fails to preserve
\begin{enumerate}[\rm(i)]
\item $\Btwo$ iff it contains either $\Neither$ or $\Both$,
\item $\Kthree$ iff it contains either $\Both$ or $\idntob$ or $\dual$,
\item $\Pthree$ iff it contains either $\Neither$ or $\idbton$ or $\dual$.
\end{enumerate}
\end{theorem}

\begin{proof}
  (i) If $f(\tuple{a}) \notin \Btwo$ for some $\tuple{a} \subseteq \Btwo$, then substituting the appropriate constants $\True$ for each $x_{i}$ such that $a_{i} = \True$ and $\False$ for each $x_{i}$ such that $a_{i} = \False$ in $f(\tuple{x})$ yields either the constant $\Neither$ or the constant $\Both$. 

  (ii) Let $f$ be a De~Morgan function which does not preserve $\Kthree$. Then $f(\tuple{a}) = \Both$ for some $\tuple{a} \subseteq \Kthree$. Substituting $\True$ for each $x_{i}$ such that $a_{i} = \True$, $\False$ for each $x_{i}$ such that $a_{i} = \False$ in $f(\tuple{x})$, and $x$ for each $x_{i}$ such that $a_{i} = \Neither$ yields either the constant $\Both$ or a unary function $f$ such that $f(\Neither) = \Both$.

  Now consider the function $g(x) = f(x \wedge \dmneg x) \wedge \dmneg f(x \wedge \dmneg x)$. Clearly $g(\Neither) = \Both$ and moreover $g(\True) = g(\False)$ and $g(a) < \True$ for each $a \in \DMfour$. If $g(\True) = \Both$, then our clone contains $\Both$. If $g(\True) = \Neither$, then $g(g(\True)) = \Both$, thus our clone again contains~$\Both$. Finally, suppose that $g(\True) = g(\False) = \False$. If $g(\Both) = \Neither$, then $\dual x = g(x) \vee (x \wedge \dmneg g(x))$. If $g(\Both) \leq \Both$, then $\idntob x = g(x) \vee (x \wedge \dmneg g(x))$.

  (iii) This follows from (ii) by conflation symmetry.
\end{proof}

  Recall that $\Truetton x$ is defined as $\Neither$ if $x = \True$ and as $\True$ otherwise, and similarly for $\Truettob$. In other words, $\Truetton x = \Neither \vee \dmneg \Box x$ and $\Truettob x = \Both \vee \dmneg \Box x$.

\begin{lemma}[Unary non-harmonious functions] \label{unary non-harmonious functions}
  A De~Morgan clone above $\DMAClone$ contains a unary non-harmonious function if and only if it contains $\Truenton$ or $\Truebtob$.
\end{lemma}

\begin{proof}
  Let $\idntof x$ be the unary De~Morgan function such that $\idntof \Neither = \False$ and $\idntof x = x$ otherwise. It follows that $\Truebtob = \idntof x \vee \dmneg \idntof x$.

  Let $f$ be a unary function such that $f(\Neither) \neq \dual f(\Both)$. If $f$ does not preserve~$\Btwo$, then one of the constants $\Neither$ or $\Both$ is expressible as $f(\True)$ or $f(\False)$. But $\Truenton x = x \vee \dmneg x \vee \Neither$, while $\Truebtob = x \vee \dmneg x \vee \Both$.

  Suppose therefore that $f$ preserves $\Btwo$. Taking the function $\dmneg f(x)$ instead of $f$ if necessary, we may assume without loss of generality that $f(\True) = \True$.

  If $f(\Neither) = \Neither$ and $f(\Both) \neq \False$, then $\Truenton x = x \vee \dmneg x \vee f(x)$. If $f(\Neither) = \Neither$ and $f(\Both) = \False$, consider instead the function $\dmneg f(x)$.

  If $f(\Neither) = \Both$ and $f(\Both) \neq \False$, then $\idntof x = x \wedge f(x)$. If $f(\Neither) = \Both$ and $f(\Both) = \False$, then $\idntof x = (x \wedge f(x)) \vee (x \wedge \dmneg f(x))$.

  If $f(\Neither) = \False$ and $f(\Both) \neq \Neither$, then $\idntof x = x \wedge f(x)$. If $f(\Neither) = \False$ and $f(\Both) = \Neither$, then considering instead the function $x \vee f(x)$ reduces the situation to the case $f(\Neither) = \Neither$ and $f(\Both) = \True$.

  Finally, if $f(\Neither) = \True$ and $f(\Both) = \Both$, then $\Truebtob x = f(x) \vee \dmneg f(x)$. If $f(\Neither) = \True$ and $f(\Both) \neq \Both$, then considering instead the function $x \wedge f(x)$ reduces the situation to the case $f(\Neither) = \Neither$ and $f(\Both) = \False$.
\end{proof}

  The binary non-harmonious functions $\mnh_{1}$, $\mnh_{2}$, $\nh_{1}$, \dots, $\nh_{6}$ used in the following theorem are defined in Figures~\ref{fig: non-harmonious binary functions} and \ref{fig: other non-harmonious binary functions}. Observe that $\mnh_{1}$ and $\mnh_{2}$ are conflation duals. (The name stands for ``minimal non-harmonious.'')

\begin{figure}
\begin{center}
\begin{tabular}{r | c c c c }
  $\mnh_{1}$ & $\True$ & $\False$ & $\Neither$ & $\Both$ \\
  \hline
  $\True$ & $\False$ & $\False$ & $\Neither$ & $\Both$ \\
  $\False$ & $\False$ & $\False$ & $\Neither$ & $\Both$ \\
  $\Neither$ & $\False$ & $\False$ & $\Neither$ & $\False$ \\
  $\Both$ & $\False$ & $\False$ & $\Neither$ & $\Both$
\end{tabular}
\qquad
\begin{tabular}{r | c c c c }
  $\mnh_{2}$ & $\True$ & $\False$ & $\Neither$ & $\Both$ \\
  \hline
  $\True$ & $\False$ & $\False$ & $\Neither$ & $\Both$ \\
  $\False$ & $\False$ & $\False$ & $\Neither$ & $\Both$ \\
  $\Neither$ & $\False$ & $\False$ & $\Neither$ & $\Both$ \\
  $\Both$ & $\False$ & $\False$ & $\False$ & $\Both$
\end{tabular}
\end{center}
\caption{The two minimal non-harmonious functions}
\label{fig: non-harmonious binary functions}
\end{figure}

\begin{figure}
\begin{center}
\begin{tabular}{r | c c c c }
  $\nh_{1}$ & $\True$ & $\False$ & $\Neither$ & $\Both$ \\
  \hline
  $\True$ & $\False$ & $\False$ & $\False$ & $\False$ \\
  $\False$ & $\False$ & $\False$ & $\False$ & $\False$ \\
  $\Neither$ & $\False$ & $\False$ & $\False$ & $\False$ \\
  $\Both$ & $\False$ & $\False$ & $\Neither$ & $\False$
\end{tabular}
\qquad
\begin{tabular}{r | c c c c }
  $\nh_{2}$ & $\True$ & $\False$ & $\Neither$ & $\Both$ \\
  \hline
  $\True$ & $\False$ & $\False$ & $\False$ & $\False$ \\
  $\False$ & $\False$ & $\False$ & $\False$ & $\False$ \\
  $\Neither$ & $\False$ & $\False$ & $\Neither$ & $\False$ \\
  $\Both$ & $\False$ & $\False$ & $\Neither$ & $\Both$
\end{tabular}

\bigskip

\begin{tabular}{r | c c c c }
  $\nh_{3}$ & $\True$ & $\False$ & $\Neither$ & $\Both$ \\
  \hline
  $\True$ & $\False$ & $\False$ & $\False$ & $\False$ \\
  $\False$ & $\False$ & $\False$ & $\False$ & $\False$ \\
  $\Neither$ & $\False$ & $\False$ & $\False$ & $\Both$ \\
  $\Both$ & $\False$ & $\False$ & $\False$ & $\False$
\end{tabular}
\qquad
\begin{tabular}{r | c c c c }
  $\nh_{4}$ & $\True$ & $\False$ & $\Neither$ & $\Both$ \\
  \hline
  $\True$ & $\False$ & $\False$ & $\False$ & $\False$ \\
  $\False$ & $\False$ & $\False$ & $\False$ & $\False$ \\
  $\Neither$ & $\False$ & $\False$ & $\Neither$ & $\Both$ \\
  $\Both$ & $\False$ & $\False$ & $\False$ & $\Both$
\end{tabular}
\end{center}
\caption{Some other non-harmonious functions}
\label{fig: other non-harmonious binary functions}
\end{figure}

\begin{theorem}[Non-harmonious clones] \label{non-harmonious clones}
  A De Morgan clone above $\DMAClone$ is non-harmonious if and only if it contains $\mnh_{1}$ or $\mnh_{2}$.
\end{theorem}

\begin{proof}
  Suppose that a clone above $\DMAClone$ contains a non-harmonious function~$f$. Then $f(\dual \tuple{a}) \neq \dual f(\tuple{a})$ for some $\tuple{a} \subseteq \DMfour$. Substituting $\True$ for each $x_{i}$ such that $a_{i} = \True$, $\False$ for each $x_{i}$ such that $a_{i} = \False$, $x$ for each $x_{i}$ such that $a_{i} = \Neither$ and $y$ for each $x_{i}$ such that $a_{i} = \Both$ yields either one of the constants $\Neither$, $\Both$ or a unary function $f$ such that $f(\Neither) \neq \dual f(\Both)$ or a binary function $f$ such that $f(\Neither, \Both) \neq \dual f(\Both, \Neither)$. The first two cases have already been dealt with in Lemma~\ref{unary non-harmonious functions}. Let us therefore turn to the binary case now and assume that $\dual f(a, b) = f(a, \dual b)$ and $\dual f(b, a) = f(\dual b, a)$ for each $a \in \Btwo$ and $b \in \DMfour$, and moreover $\dual f(a, a) = f(\dual a, \dual a)$ for each $a \in \DMfour$.

  We may assume without loss of generality that $f(\Neither, \Both) = \False$. If $f(\Neither, \Both) = \True$, take the function $\dmneg f(x, y)$ instead. If $f(\Neither, \Both) = \Neither$ and $f(\Both, \Neither) \geq \Neither$, take the function $f(x, y) \wedge y$ instead. If $f(\Neither, \Both) = \Neither$ and $f(\Both, \Neither) = \False$, take the function $f(y, x)$ instead. If $f(\Neither, \Both) = \Both$ and $f(\Both, \Neither) \geq \Both$, take the function $f(x, y) \wedge x$ instead. If $f(\Neither, \Both) = \Both$ and $f(\Both, \Neither) = \False$, take the function $f(y, x)$ instead.

  Instead of $f$, consider now the binary function
\begin{align*}
  g(x, y) & = f(x \wedge \dmneg x, y \wedge \dmneg y) \wedge (x \vee y) \wedge (x \vee \dmneg y) \wedge (\dmneg x \vee y) \wedge (\dmneg x \vee \dmneg y).
\end{align*}
  This function satisfies $g(\Neither, \Both) = \False$ and $g(\Both, \Neither) \neq \dual g(\Neither, \Both) = \False$. For each $a \in \Btwo$ we also have
\begin{align*}
  g(\Neither, a) & \leq \Neither, & g(a, \Neither) & \leq \Neither, \\
  g(\Both, a) & \leq \Both, & g(a, \Both) & \leq \Both.
\end{align*}
  Moreover, $g(a, b) = \False$ whenever $\{ a, b \} \subseteq \Btwo$, and $g(\False, a) = g(\True, a)$ and $g(a, \False) = g(a, \True)$ for each $a \in \DMfour$. Finally, we may assume that $g(\Neither, \Neither) = \dual g(\Both, \Both)$ and $g(\True, \dual a) = \dual g(\True, a)$ and $g(\False, \dual a) = \dual g(\False, a)$, otherwise taking the function $g(x, x)$ or $g(\True, a)$ or $g(\False, a)$ reduces the situation to the unary case.

  There are eight cases to consider now.
\begin{itemize}
\item[(a)] Suppose that $g(\Both, \Neither) \geq \Neither$. Then consider $h(x, y) \assign f(x, y) \wedge y \wedge \dmneg y$. We have $h(a, b) = \False$ for each $b \in \Btwo$ and $h(a, b) \leq b$ for each $b \in \DMfour$. Moreover, $h(\Neither, \Both) = \False$ and $h(\Both, \Neither) = \Neither$. It remains to consider four subcases.
  \begin{itemize}
    \item[(i)] $h(a, b) = \False$ and $h(b, b) = \False$ for $a \in \{ \True, \False \}$ and $b \in \{ \Neither, \Both \}$.
    \item[(ii)] $h(a, b) = \False$ and $h(b, b) = b$ for $a \in \{ \True, \False \}$ and $b \in \{ \Neither, \Both \}$.
    \item[(iii)] $h(a, b) = b$ and $h(b, b) = \False$ for $a \in \{ \True, \False \}$ and $b \in \{ \Neither, \Both \}$. But now the function $h(h(y, x), x)$ is the function $h$ from case (b ii).
    \item[(iv)] $h(a, b) = b$ and $h(b, b) = b$ for $a \in \{ \True, \False \}$ and $b \in \{ \Neither, \Both \}$.
  \end{itemize}
\item[(b)] Suppose that $g(\Both, \Neither) \geq \Both$. Then consider $h(x, y) = f(x, y) \wedge x \wedge \dmneg x$. We have $h(a, b) = \False$ for each $a \in \Btwo$ and $h(a, b) \leq a$ for each $a \in \DMfour$. Moreover, $h(\Neither, \Both) = \False$ and $h(\Both, \Neither) = \Both$. It remains to consider four subcases.
  \begin{itemize}
    \item[(i)] $h(b, a) = \False$ and $h(b, b) = \False$ for $a \in \{ \True, \False \}$ and $b \in \{ \Neither, \Both \}$.
    \item[(ii)] $h(b, a) = \False$ and $h(b, b) = b$ for $a \in \{ \True, \False \}$ and $b \in \{ \Neither, \Both \}$.
    \item[(iii)] $h(b, a) = b$ and $h(b, b) = \False$ for $a \in \{ \True, \False \}$ and $b \in \{ \Neither, \Both \}$. But now the function $h(y, h(y, x))$ is the function $h$ from case (a ii).
    \item[(iv)] $h(b, a) = b$ and $h(b, b) = b$ for $a \in \{ \True, \False \}$ and $b \in \{ \Neither, \Both \}$.
  \end{itemize}
\end{itemize}
  The above case analysis shows that one of the functions $\Truenton$, $\Truebtob$, $\mnh_{1}$, $\mnh_{2}$, $\nh_{1}$, \dots, $\nh_{4}$ belongs to the clone in question.

  It remains to show that $\mnh_{1}$ and $\mnh_{2}$ can be expressed in terms of the other functions:
\begin{align*}
  \mnh_{1}(x, y) & = y \wedge \dmneg y \wedge \Truenton x, \\
  \nh_{2}(x, y) & = \nh_{1}(x, y) \vee (x \wedge \dmneg x \wedge y \wedge \dmneg y), \\
  \mnh_{1}(x, y) & = (y \wedge \dmneg y) \wedge [((x \vee \dmneg x) \wedge (y \vee \dmneg y)) \vee \nh_{2}(x, y)],
\end{align*}
  The same relations hold between $\mnh_{2}$, $\Truebtob$, $\nh_{3}$, and $\nh_{4}$, since these are, respectively, the conflation duals of $\mnh_{1}$, $\Truenton$, $\nh_{1}$, and $\nh_{2}$.
\end{proof}

\begin{lemma}[Unary non-persistent functions] \label{unary non-persistent functions}
  A De~Morgan clone above $\DMAClone$ contains a unary non-persistent function if and only if it contains one of the functions $\Box$, $\idntot$, $\idbtot$. It contains a unary harmonious non-persistent function if and only if it contains~$\Box$.
\end{lemma}

\begin{proof}
  Let $f$ be a unary non-persistent function. Then either $f(\Neither) \ninfleq f(\True)$ or $f(\Neither) \ninfleq f(\False)$ or $f(\True) \ninfleq f(\Both)$ or $f(\False) \ninfleq f(\Both)$. Taking $f(\dmneg x)$ instead of $f(x)$ if necessary, we may assume that either $f(\Neither) \ninfleq f(\True)$ or $f(\True) \ninfleq f(\Both)$. Since these cases are related by conflation symmery, we only deal with the former case.

  If $f(\Neither) \ninfleq f(\True)$, then $f(\Neither) \neq \Neither$ and $f(\True) \neq \Both$. We may assume that $f(\Neither) < \True$ and $f(\True) > \False$, taking $\dmneg f(x)$ instead of $f(x)$ if $f(\Neither) = \True$, in which case $f(\True) < \True$. Thus $f(\Neither) \leq \Both$ and $f(\True) \geq \Neither$. But then $g(x) \assign f(x) \wedge f(\True) \wedge x$ is a non-persistent function such that $g(\Neither) = g(\False) = \False$ and $g(\True) \geq \Neither$ and $g(\Both) \leq \Both \wedge g(\True)$. If $f$ is harmonious, then so is $g$, hence $g(\True) = \True$ and $g(\Both) = \False$, i.e.\ $g(x) = \Box x$. If $g(\True) = \Neither$, then $g(\Both) = \False$, so $\dmneg g(x) = \Truetton x$. Finally, if $g(\True) = \True$ and $g(\Both) = \Both$, then $\dmneg g(\dmneg x) = \idntot x$. The conflation duals of $\Box$, $\Truetton$, and $\idntot$ are the functions $\Box$, $\Truettob$, $\idbtot$. Finally, observe that
\begin{align*}
  \idbtot x & = \dmneg (\Truetton \Truetton \dmneg (x \wedge \Truetton x)) \vee x.
\end{align*}
  By conflation duality, $\idntot$ can be expressed similarly in terms of $\Truettob$.
\end{proof}

  The binary functions $\mhnp$, $\mnp_{1}$, \dots, $\mnp_{4}$, $\np_{1}$, $\np_{2}$, $\np_{3}$ used in the following theorem are defined in Figures~\ref{fig: non-persistent binary functions} and \ref{fig: other non-persistent binary functions}. Observe that $\mnp_{1}$ and $\mnp_{3}$ are conflation duals, as are $\mnp_{2}$ and $\mnp_{4}$.

\begin{figure}
\begin{center}
\begin{tabular}{r | c c c c }
  $\mhnp$ & $\True$ & $\False$ & $\Neither$ & $\Both$ \\
  \hline
  $\True$ & $\True$ & $\True$ & $\True$ & $\True$ \\
  $\False$ & $\True$ & $\True$ & $\True$ & $\True$ \\
  $\Neither$ & $\Neither$ & $\Neither$ & $\Neither$ & $\False$ \\
  $\Both$ & $\Both$ & $\Both$ & $\False$ & $\Both$
\end{tabular}

\bigskip

\begin{tabular}{r | c c c c }
  $\mnp_{1}$ & $\True$ & $\False$ & $\Neither$ & $\Both$ \\
  \hline
  $\True$ & $\True$ & $\True$ & $\True$ & $\Both$ \\
  $\False$ & $\True$ & $\True$ & $\True$ & $\Both$ \\
  $\Neither$ & $\Neither$ & $\Neither$ & $\Neither$ & $\False$ \\
  $\Both$ & $\Both$ & $\Both$ & $\False$ & $\Both$
\end{tabular}
\qquad
\begin{tabular}{r | c c c c }
  $\mnp_{2}$ & $\True$ & $\False$ & $\Neither$ & $\Both$ \\
  \hline
  $\True$ & $\True$ & $\True$ & $\True$ & $\True$ \\
  $\False$ & $\True$ & $\True$ & $\True$ & $\True$ \\
  $\Neither$ & $\Neither$ & $\Neither$ & $\Neither$ & $\Neither$ \\
  $\Both$ & $\Both$ & $\Both$ & $\False$ & $\Both$
\end{tabular}

\bigskip

\begin{tabular}{r | c c c c }
  $\mnp_{3}$ & $\True$ & $\False$ & $\Neither$ & $\Both$ \\
  \hline
  $\True$ & $\True$ & $\True$ & $\Neither$ & $\True$ \\
  $\False$ & $\True$ & $\True$ & $\Neither$ & $\True$ \\
  $\Neither$ & $\Neither$ & $\Neither$ & $\Neither$ & $\False$ \\
  $\Both$ & $\Both$ & $\Both$ & $\False$ & $\Both$
\end{tabular}
\qquad
\begin{tabular}{r | c c c c }
  $\mnp_{4}$ & $\True$ & $\False$ & $\Neither$ & $\Both$ \\
  \hline
  $\True$ & $\True$ & $\True$ & $\True$ & $\True$ \\
  $\False$ & $\True$ & $\True$ & $\True$ & $\True$ \\
  $\Neither$ & $\Neither$ & $\Neither$ & $\Neither$ & $\False$ \\
  $\Both$ & $\Both$ & $\Both$ & $\Both$ & $\Both$ \\
\end{tabular}
\end{center}
\caption{The minimal non-persistent binary functions}
\label{fig: non-persistent binary functions}
\end{figure}

\begin{figure}
\begin{center}
\begin{tabular}{r | c c c c }
  $\np_{1}$ & $\True$ & $\False$ & $\Neither$ & $\Both$ \\
  \hline
  $\True$ & $\False$ & $\False$ & $\Neither$ & $\Both$ \\
  $\False$ & $\False$ & $\False$ & $\False$ & $\False$ \\
  $\Neither$ & $\False$ & $\False$ & $\False$ & $\False$ \\
  $\Both$ & $\False$ & $\False$ & $\False$ & $\False$
\end{tabular}
\qquad
\begin{tabular}{r | c c c c }
  $\np_{2}$ & $\True$ & $\False$ & $\Neither$ & $\Both$ \\
  \hline
  $\True$ & $\False$ & $\False$ & $\Neither$ & $\Both$ \\
  $\False$ & $\False$ & $\False$ & $\False$ & $\False$ \\
  $\Neither$ & $\False$ & $\False$ & $\False$ & $\False$ \\
  $\Both$ & $\False$ & $\False$ & $\False$ & $\Both$
\end{tabular}
\qquad
\begin{tabular}{r | c c c c }
  $\np_{3}$ & $\True$ & $\False$ & $\Neither$ & $\Both$ \\
  \hline
  $\True$ & $\False$ & $\False$ & $\Neither$ & $\False$ \\
  $\False$ & $\False$ & $\False$ & $\False$ & $\False$ \\
  $\Neither$ & $\False$ & $\False$ & $\False$ & $\False$ \\
  $\Both$ & $\False$ & $\False$ & $\False$ & $\False$
\end{tabular}
\end{center}
\caption{Some other non-persistent binary functions}
\label{fig: other non-persistent binary functions}
\end{figure}

\begin{lemma}[Binary non-persistent functions] \label{binary non-persistent functions}
  A De~Morgan clone above $\DMAClone$ contains a binary non-persistent function if and only if it contains one of the functions $\mhnp$, $\mnp_{1}$, \dots, $\mnp_{4}$. It contains a binary harmonious non-persistent function if and only if contains $\mhnp$.
\end{lemma}

\begin{proof}
  Suppose that each unary function in a clone above $\DMAClone$ is persistent but the clone contains a binary non-persistent function $f(x, y)$. Then $a \infleq b$ implies $f(a, c) \infleq f(b, c)$ and $f(c, a) \infleq f(c, b)$ for $c \in \{ \True, \False \}$. It follows that there are $a \infleq b$ and $c \notin \{ \True, \False \}$ such that $f(a, c) \ninfleq f(b, c)$ or $f(c, a) \ninfleq f(c, b)$. We may assume that the former is the case, taking $f(y, x)$ instead of $f(x, y)$ if necessary. The cases $c = \Neither$ and $c = \Both$ are related by conflation symmetry, we therefore only deal with the former case. That is, we assume that $f(a, \Neither) \ninfleq f(b, \Neither)$ for some $a \infleq b$. It follows that $f(a, \Neither) \ninfleq f(b, \Neither)$ for some $a$ and $b$ such that $b$ is an upper cover of $a$ in the information order. Taking $f(\dmneg x, y)$ instead of $f(x, y)$ if necessary, we may assume that either $a = \Neither$ and $b = \True$ or $a = \True$ and $b = \Both$, i.e.\ $f(\Neither, \Neither) \ninfleq f(\True, \Neither)$ or $f(\True, \Neither) \ninfleq f(\Both, \Neither)$.

  Let us start with the former case. If $f(\Neither, \Neither) \ninfleq f(\True, \Neither)$, then either $f(\Neither, \Neither) \geq \Both$ and $f(\True, \Neither) \leq \Neither$ or $f(\Neither, \Neither) \leq \Both$ and $f(\True, \Neither) \geq \Neither$. By taking $\dmneg f(x, y)$ instead of $f(x, y)$, we may assume without loss of generality that the former is the case. But then
\begin{align*}
  g(x, y) & \assign f(x, y) \wedge f(\True, y) \wedge x \wedge y \wedge \dmneg y
\end{align*}
  is a non-persistent function such that $g(\Neither, \Neither) = \False$ and $g(\True, \Neither) = \Neither$. Moreover, $g(\False, y) = g(x, \False) = g(x, \True) = \False$ and $g(\Neither, \Both) = g(\Both, \Neither) = \False$, while $g(\Both, \Both) \leq \Both$.

  It remains to settle the values of $g(\Both, \Both) \leq g(\True, \Both) \leq \Both$. If $f$ is harmonious, then so is $g$. It follows that $g(\Both, \Both) = \False$ and $g(\True, \Both) = \Both$, i.e.\ $g(x, y) = \np_{1}(x, y)$. If $g$ is not harmonious, then it coincides with $\np_{2}$ or $\np_{3}$.

  Let us now deal with the case $f(\True, \Neither) \ninfleq f(\Both, \Neither)$. Either $f(\Both, \Neither) \leq \Neither$ and $f(\True, \Neither) \geq \Both$, or $f(\Both, \Neither) \geq \Neither$ and $f(\True, \Neither) \leq \Both$. By taking $\dmneg f(x, y)$ instead of $f(x, y)$, we may assume the former without loss of generality. But then
\begin{align*}
  g(x, y) & \assign f(x \vee \dmneg x, y \vee \dmneg y) \vee y \vee \dmneg y
\end{align*}
  is a non-persistent function such that $g(\Both, \Neither) = \Neither$, $g(\True, \Neither) = \True$, $g(x, \True) = g(x, \False) = \True$, $g(\True, y) = g(\False, y)$, $g(\Neither, \Neither) \geq \Neither$, and $g(x, \Both) \geq \Both$. Let us now take
\begin{align*}
  h(x, y) & \assign g(x, y) \wedge g(\True, y) \wedge (x \vee \dmneg x)
\end{align*}
  Then $h$ must be one of the functions $\mhnp$, $\mnp_{1}$, or $\mnp_{2}$.

  Recall that we took advantage of conflation duality in the above analysis, and $\mnp_{3}, \mnp_{4}$ are the conflation duals of $\mnp_{1}, \mnp_{2}$. The case analysis therefore shows that the clone in question contains one of the functions $\Box$, $\idntot$, $\idbtot$, $\mhnp$, $\mnp_{1}$, $\mnp_{2}$, $\mnp_{3}$, $\mnp_{4}$, or $\np_{1}$, $\np_{2}$, $\np_{3}$, or the conflation duals of $\np_{2}$, $\np_{3}$. (The function $\np_{1}$ is harmonious.) It remains to show that some of these functions may be expressed in terms of others:
\begin{align*}
  \mhnp(x, y) & = (x \vee \dmneg x) \wedge \dmneg \np_{1}((x \wedge \dmneg x) \vee (y \wedge \dmneg y), y), \\
  \mnp_{1}(x, y) & = (x \vee \dmneg x) \wedge \dmneg \np_{2}((x \wedge \dmneg x) \vee (y \wedge \dmneg y), y), \\
  \mnp_{2}(x, y) & = (x \vee \dmneg x) \wedge \dmneg \np_{3}((x \wedge \dmneg x) \vee (y \wedge \dmneg y), y).
\end{align*}
  By conflation duality, $\mnp_{3}$ and $\mnp_{4}$ can be expressed in the same way using the conflation duals of $\np_{2}$ and $\np_{3}$. Furthermore,
\begin{align*}
  \mhnp(x, y) & = (x \vee \dmneg x) \wedge \dmneg \Box ((x \wedge \dmneg x) \vee (y \wedge \dmneg y)), \\
  \mnp_{1}(x, y) & = (\idntot x \wedge \idntot \dmneg x) \wedge x \wedge y \\ & \phantom{=} ~ \wedge (\dmneg \idntot x \vee \dmneg \idntot \dmneg x \vee \dmneg \idntot y \vee \dmneg \idntot \dmneg y).
\end{align*}
  By conflation duality, $\mnp_{3}$ can be expressed in the same way using $\idbtot$.
\end{proof}

\begin{figure}
\begin{center}
\begin{tabular}{r | c c c c }
  $\mhnpt(x, \Both, y)$ & $\True$ & $\False$ & $\Neither$ & $\Both$ \\
  \hline
  $\True$ & $\Both$ & $\Both$ & $\Both$ & $\Both$ \\
  $\False$ & $\False$ & $\False$ & $\False$ & $\False$ \\
  $\Neither$ & $\False$ & $\False$ & $\False$ & $\False$ \\
  $\Both$ & $\Both$ & $\Both$ & $\False$ & $\Both$
\end{tabular}

\bigskip

\begin{tabular}{r | c c c c }
  $\mnpt_{1}(x, \Both, y)$ & $\True$ & $\False$ & $\Neither$ & $\Both$ \\
  \hline
  $\True$ & $\Both$ & $\Both$ & $\Both$ & $\Both$ \\
  $\False$ & $\False$ & $\False$ & $\False$ & $\False$ \\
  $\Neither$ & $\False$ & $\False$ & $\False$ & $\False$ \\
  $\Both$ & $\Both$ & $\Both$ & $\Both$ & $\Both$
\end{tabular}
\qquad
\begin{tabular}{r | c c c c }
  $\mnpt_{2}(x, \Both, y)$ & $\True$ & $\False$ & $\Neither$ & $\Both$ \\
  \hline
  $\True$ & $\Both$ & $\Both$ & $\False$ & $\Both$ \\
  $\False$ & $\False$ & $\False$ & $\False$ & $\False$ \\
  $\Neither$ & $\False$ & $\False$ & $\False$ & $\False$ \\
  $\Both$ & $\Both$ & $\Both$ & $\False$ & $\Both$
\end{tabular}
\end{center}
\caption{Minimal non-persistent ternary functions}
\label{fig: non-persistent ternary functions}
\end{figure}

\begin{figure}
\begin{center}
\begin{tabular}{r | c c c c }
  $\npt_{1}(x, \Both, y)$ & $\True$ & $\False$ & $\Neither$ & $\Both$ \\
  \hline
  $\True$ & $\False$ & $\False$ & $\False$ & $\Both$ \\
  $\False$ & $\False$ & $\False$ & $\False$ & $\False$ \\
  $\Neither$ & $\False$ & $\False$ & $\False$ & $\False$ \\
  $\Both$ & $\False$ & $\False$ & $\False$ & $\Both$
\end{tabular}
\qquad
\begin{tabular}{r | c c c c }
  $\npt_{2}(x, \Both, y)$ & $\True$ & $\False$ & $\Neither$ & $\Both$ \\
  \hline
  $\True$ & $\False$ & $\False$ & $\False$ & $\False$ \\
  $\False$ & $\False$ & $\False$ & $\False$ & $\False$ \\
  $\Neither$ & $\False$ & $\False$ & $\False$ & $\False$ \\
  $\Both$ & $\False$ & $\False$ & $\False$ & $\False$
\end{tabular}
\end{center}
\caption{Some other non-persistent ternary functions}
\label{fig: other non-persistent ternary functions}
\end{figure}

  At the level of ternary functions, the combinatorics gets quite messy, therefore we shall rely on some computer assistance. In particular, we shall use the Baker--Pixley theorem mentioned in the introduction, which implies that $g \in \clonegen{\DMAClone, f}$ if and only if each subalgebra of $\DMfouralg \times \DMfouralg$ preserved by $f$ is also preserved by $g$. This claim about subalgebras may then be checked mechanically by a computer. (A Haskell script written for this purpose can be found on the author's webpage.)  The author has in fact already covertly been using a computer to assist him in determining which of the clones used in the proofs above are comparable and which are not.

  Let us now define the ternary functions $\mhnpt$, $\mnpt_{1}$, $\mnpt_{2}$, $\npt_{1}$, $\npt_{2}$ used in the following theorem. The function $\mhnpt$ is the unique ternary function such that
\begin{align*}
  \mhnpt(x, \True, z) & = \False, & \! \mhnpt(\Neither, \Neither, \Both) & = \False, & \! \mhnpt(x, \Neither, z) & = \False \text{ for } x \in \{ \False, \Both \}, \\
  \mhnpt(x, \False, z) & = \False, & \! \mhnpt(\True, \Neither, z) & = \Neither, & \! \mhnpt(\Neither, \Neither, z) & = \Neither \text{ for } z \in \{ \True, \Neither, \False \},
 \end{align*}
  and $\mhnpt(x, \Both, y)$ is the binary function shown in Figure~\ref{fig: non-persistent ternary functions}. The other ternary functions then agree with $\mhnpt$ everywhere except for the values shown in Figures~\ref{fig: non-persistent ternary functions} and \ref{fig: other non-persistent ternary functions}. Out of these functions only $\mhnpt$ is harmonious.

\begin{theorem}[Non-persistent clones] \label{non-persistent clones}
  A De Morgan clone above $\DMAClone$ contains a non-persistent function if and only if it contains one of the ternary functions $\mhnpt$, $\mnpt_{1}$, $\mnpt_{2}$ or the conflation dual of $\mnpt_{1}$ or $\mnpt_{2}$. A~De~Morgan clone above $\DMAClone$ contains a harmonious non-persistent function if and only if it contains $\mhnpt$.
\end{theorem}

\begin{proof}
  Let $f$ be a non-persistent function in a clone above $\DMAClone$ such that all binary functions in this clone are persistent. Then there are $\tuple{a}$ and $\tuple{b}$ such that $\tuple{a} \infleq \tuple{b}$ and $f(\tuple{a}) \ninfleq f(\tuple{b})$. It follows that there are some $a, b \in \DMfour$ and some $\tuple{c} \subseteq \DMfour$ such that $a \infleq b$ and $f(a, \tuple{c}) \ninfleq f(b, \tuple{c})$ (rearranging the order of arguments if necessary). We may assume that $\tuple{c} \subseteq \{ \Neither, \Both \}$ by substituting $\True$ for each $y_{i}$ such that $c_{i} = \True$ and $\False$ for each $y_{i}$ such that $c_{i} = \False$ in $f(x, \tuple{y})$. Substituting $y$ for each $y_{i}$ such that $c_{i} = \Neither$ and $z$ for each $y_{i}$ such that $c_{i} = \Both$ now yields a ternary function $f(x, y, z)$ such that $f(a, \Neither, \Both) \ninfleq f(b, \Neither, \Both)$ for some $a, b \in \DMfour$ such that $a \infleq b$. We may assume without loss of generality that $b$ is a cover of $a$ in the information order.

  By taking $f(\dmneg x, y, z)$ instead of $f(x, y, z)$ if necessary, we may assume without loss of generality that either $a = \Neither$ and $b = \True$ or $a = \True$ and $b = \Both$. Since these two cases are related by conflation symmetry, we focus only on the first case. That is, we shall assume that $f(\Neither, \Neither, \Both) \ninfleq f(\True, \Neither, \Both)$.

  Clearly $f(\Neither, \Neither, \Both) \neq \Neither$ and $f(\True, \Neither, \Both) \neq \Both$. We now show that we may assume without loss of generality that $f(\Neither, \Neither, \Both) = \False$ and $f(\True, \Neither, \Both) = \Neither$: first, we may assume that $f(\Neither, \Neither, \Both) < \True$, by taking $\dmneg f(x, y, z)$ instead of $f(x, y, z)$. Moreover, if $f(\Neither, \Neither, \Both) = \Both$ and $f(\True, \Neither, \Both) \geq \Neither$, consider $f(x, y, z) \wedge y$ instead. If $f(\Neither, \Neither, \Both) = \Both$ and $f(\True, \Neither, \Both) = \False$, consider $\dmneg f(x, y, z) \wedge y$ instead. If $f(\Neither, \Neither, \Both) = \False$, then $f(\True, \Neither, \Both) \geq \Neither$, and we consider $f(x, y, z) \wedge y$ instead.

  Now consider instead of $f$ the function
\begin{align*}
  g(x, y, z) & \assign f(x, y, z) \wedge f(x, y, \dmneg z) \wedge x \wedge y \wedge \dmneg y \wedge f(\True, y, z) \wedge f(\True, y, \dmneg z).
\end{align*}
  We have $g(x, \True, z) = g(x, \False, z) = \False$ and $g(\False, y, z) = \False$ and $g(\Neither, \Both, z) = g(\Both, \Neither, z) = \False$. We also know that $g(\Neither, \Neither, \Both) = \False$ and $g(\True, \Neither, \Both) = \Neither$. It remains to settle the values of $g(\True, y, z)$ for $y \in \{ \Neither, \Both \}$ and $g(\Neither, \Neither, z)$ and $g(\Both, \Both, z)$.

  We know that $g(\Neither, \Neither, z) \leq g(\True, \Neither, z) \leq \Neither$ and $g(\Both, \Both, z) \leq g(\True, \Both, z) \leq \Both$ and moreover $g(x, y, \True) = g(x, y, \False)$. The function $g(\True, x, y)$ is persistent because each binary function in the clone in question is persistent, thus $g(\True, \Neither, \Both) = \Neither$ implies that $g(\True, \Neither, z) = \Neither$ for each~$z$. The function $g(x, y, y)$ is also persistent, thus $g(\True, \Neither, \Neither) = \Neither$ implies $g(\Neither, \Neither, \Neither) = \Neither$. The function $g(x, y, \True)$ is persistent, thus $g(\True, \Neither, \True) = \Neither$ implies $g(\Neither, \Neither, \True) = \Neither$. If $f$ is harmonious, then so is $g$, in which case the above uniquely determines $g$.
  
  To settle the values of $g(\True, \Both, z)$ and $g(\Both, \Both, z)$, recall that $g(x, y, \True) = g(x, y, \False)$ and $g(\Both, \Both, z) \leq g(\True, \Both, z) \leq \Both$. The functions $g(x, y, \True)$ and $g(x, y, y)$ are persistent, thus $g(\True, \Both, \True) \infleq g(\Both, \Both, \True)$ and $g(\True, \Both, \Both) \infleq g(\Both, \Both, \Both)$. It follows that $g(\True, \Both, \True) = g(\Both, \Both, \True)$ and $g(\True, \Both, \Both) = g(\Both, \Both, \Both)$. If $g(\Both, \Both, \Neither) = \Both$ (which must be the case if $f$ is harmonious), then $g(\Both, \Both, \True) = g(\Both, \Both, \Both) = \Both$ by the monotonicity of $g(x, x, y)$ and $\Both = g(\Both, \Both, \Neither) \leq g(\True, \Both, \Neither)$.

  Otherwise, suppose that $g(\Both, \Both, \Neither) = \False$. If $g(\True, \Both, \Neither) = \Both$, then $g(\True, \Both, \True) = g(\True, \Both, \Both) = \Both$ by the monotonicity of $g(\True, x, y)$. If on the other hand $g(\True, \Both, \Neither) = \False$, then three options remain: either $g(\True, \Both, \True) = g(\True, \Both, \Both) = \Both$, or $g(\True, \Both, \True) = \False$ and $g(\True, \Both, \Both) = \Both$, or $g(\True, \Both, \True) = g(\True, \Both, \Both) = \False$.

  The above case analysis shows that the clone in question must contain one of the binary functions $\mhnp$, $\mnp_{1}$, \dots, $\mnp_{4}$ or the ternary function $\mhnpt$ or one of the ternary functions $\mnpt_{1}$, $\mnpt_{2}$, $\npt_{1}$, $\npt_{2}$ or their conflation duals. We may now verify using a computer that in fact $\mnpt_{2}$ belongs to $\clonegen{\DMAClone, \npt_{1}}$ as well as $\clonegen{\DMAClone, \npt_{2}}$. Moreover, $\mhnpt$ lies in $\clonegen{\DMAClone, \mhnp}$, $\mnpt_{1}$ in $\clonegen{\DMAClone, \mnp_{1}}$, $\mnpt_{2}$ in $\clonegen{\DMAClone, \mnp_{2}}$, and accordingly the conflation duals of $\mnpt_{1}$ and $\mnpt_{2}$ lie in $\clonegen{\DMAClone, \mnp_{3}}$ and $\clonegen{\DMAClone, \mnp_{4}}$, respectively.
\end{proof}

  The previous sequence of results now allows us to show that $\DMAClone$ has exactly three covers in the lattice of De~Morgan clones.

\begin{theorem}[Covers of $\DMAClone$] \label{covers of dma}
  Each De~Morgan clone strictly above $\DMAClone$ contains either $\mnh_{1}$, $\mnh_{2}$, $\mhnpt$. Moreover, the clones $\clonegen{\DMAClone, \mhnpt}$ $\clonegen{\DMAClone, \mnh_{1}}$, $\clonegen{\DMAClone, \mnh_{2}}$ are incomparable by inclusion.
\end{theorem}

\begin{proof}
  Consider a De~Morgan clone strictly above $\DMAClone$. Because $\DMAClone$ is the clone of all persistent harmonious De~Morgan functions by Theorem~\ref{harmonious clones}, this clone contains either a non-harmonious function or a harmonious non-persistent function. The former case is covered by our description of non-harmonious clones (Theorem~\ref{non-harmonious clones}). The latter case is covered our description of harmonious non-persistent clones (Theorem~\ref{non-persistent clones}). It remains to show that the three clones are incomparable by inclusion. To this end, it suffices to exhibit a subalgebra of $\DMfouralg \times \DMfouralg$ which is closed under $\mnh_{1}$ but not under $\mnh_{2}$, and likewise for the other five pairs. In each case such a subalgebra exists, although the task of finding it is best left to a computer.
\end{proof}

\begin{corollary}
 $\DMAClone$ has exactly three covers in the lattice of De~Morgan clones, namely $\clonegen{\DMAClone, \mnh_{1}}$, $\clonegen{\DMAClone, \mnh_{2}}$, and $\clonegen{\DMAClone, \mhnpt}$.
\end{corollary}

\section{Discriminator clones}
\label{sec: discriminator}

  We now use the results of the previous section to describe the lattice of all De~Morgan clones above $\DMAClone$ which contain the discriminator function. Recall that the \emph{quaternary discriminator} on a given set is the quaternary function $d(x, y, z, u)$ such that
\begin{align*}
  d(x, y, z, u) & = \begin{cases} z \text{ if } x = y, \\ u \text{ if } x \neq y. \end{cases}
\end{align*}
  The De~Morgan discriminator is known to be inter\-definable with $\Box x$ over the clone~$\DMAClone$: clearly $\Box x = d(x, \True, \True, \False)$. Conversely, 
\begin{align*}
  d(x, y, z, u) & = ((x \equivcrisp y) \wedge z) \vee (\dmneg (x \equivcrisp y) \wedge u),
\end{align*}
  where
\begin{align*}
  x \equivcrisp y & = \Box (x \wedge y) \vee \Box (\dmneg x \wedge \dmneg y) \vee (\dmneg \Box (x \vee y) \wedge \dmneg \Box (\dmneg x \vee \dmneg y)).
\end{align*}
  Observe that $a \equivcrisp b = \True$ if $a = b$ and $a \equivcrisp b = \False$ if $a \neq b$. Our goal is therefore to describe the lattice of De~Morgan clones above $\clonegen{\DMAClone, \Box}$.

  In order to describe this lattice, we first need to consider one more clone which lies between $\clonegen{\DMAClone, \Box}$ and $\clonegen{\DMAClone, \Delta}$.

\begin{definition}[Partially harmonious functions]
  A De Morgan function $f$ is \emph{partially harmonious} if $f(\dual \tuple{a}) = \dual f(\tuple{a})$ whenever $\tuple{a} \subseteq \Kthree$ or $\tuple{a} \subseteq \Pthree$.
\end{definition}

  An example of a partially harmonious function which is not harmonious is the binary function $\Deltanbpair(x, y)$ such that
\begin{align*}
  \Deltanbpair(a, b) = \True \text{ if } a = \Neither \text{ and } b = \Both, & & \Deltanbpair(a, b) = \False \text{ otherwise}.
\end{align*}
  Beware that the set of all partially harmonious De~Morgan functions does not form a clone: the functions $\Deltanbpair(x, y)$, $\Box x$, $\dual x$, $x \vee y$ are partially harmonious, but $\Delta x = \Box x \vee \Deltanbpair(\dual x, x)$ is not. However, partially harmonious functions which moreover preserve $\Btwo$, $\Kthree$, and $\Pthree$ do form a clone.

\begin{figure}
\begin{center}
\begin{tikzpicture}[scale=1.5]
  \node (Box) at (0,0) {$\Box$};
  \node (Deltanb) at (0,1) {$\Box, \hbox to 1em{$\Deltanbpair$\hss}$};
  \node (Delta) at (0,2) {$\Delta$};
  \node (idbton) at (-1,3) {$\Box, \idbton$};
  \node (idntob) at (1,3) {$\Box, \idntob$};
  \node (n) at (-1,4) {$\Box, \Neither$};
  \node (b) at (1,4) {$\Box, \Both$};
  \node (Deltadual) at (0,4) {$\Delta, \dual$};
  \node (top) at (0,5) {$\Box, \Neither, \Both$};
  \node (delta) at (-1,2) {$\dual$};
  \draw[-] (Box) -- (Deltanb) -- (Delta) -- (idbton) -- (n) -- (top);
  \draw[-] (Delta) -- (idntob) -- (b) -- (top);
  \draw[-] (idbton) -- (Deltadual) -- (top);
  \draw[-] (idntob) -- (Deltadual);
  \draw[-] (Box) -- (delta) -- (Deltadual);
\end{tikzpicture}
\end{center}
\caption{De Morgan clones above $\DMAClone$ which contain the De~Morgan discriminator}
\label{fig: discriminator clones}
\end{figure}
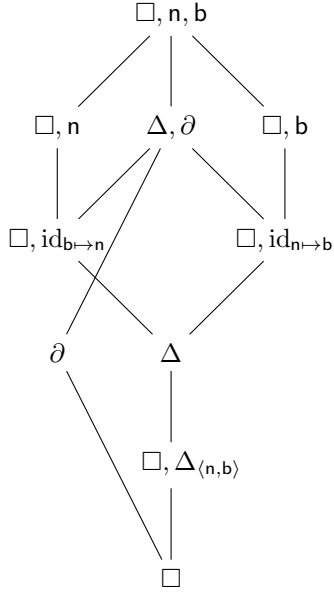

\begin{theorem}[The partially harmonious subalg.-preserving clone] \label{partially harmonious clone}
  The clone $\clonegen{\DMAClone, \Box, \Deltanbpair}$ is the clone of all partially harmonious De Morgan functions which preserve $\Btwo$, $\Kthree$, and $\Pthree$.
\end{theorem}

\begin{proof}
  If $\{ \Neither, \Both \} \subseteq \tuple{a}$, then the smallest partially harmonious function $f_{\tuple{a}, b}$ which preserves $\Btwo$, $\Kthree$, $\Pthree$ such that $f_{\tuple{a}, b}(\tuple{a}) = b$ is in fact the smallest De~Morgan function $f_{\tuple{a}, b}$ such that $f_{\tuple{a}, b}(\tuple{a}) = b$, namely
\begin{align*}
  f_{\tuple{a}, b}(\tuple{x}) = b & \text{ if } \tuple{x} = \tuple{a}, \\
  f_{\tuple{a}, b}(\tuple{x}) = \False & \text{ otherwise}.
\end{align*}
  It is easy to see that $f_{\tuple{a}, \True}$ can be expressed as
\begin{align*}
  f_{\tuple{a}, \True}(\tuple{x}) = \smashoperator{\bigwedge_{a_{i} = \True}} \Box x_{i} ~ \wedge ~ \smashoperator{\bigwedge_{a_{i} = \False}} \Box \dmneg x_{i} ~ \wedge ~ \smashoperator{\bigwedge_{\pair{a_{i}}{a_{j}} = \pair{\Neither}{\Both}}} \Deltanbpair(x_{i}, x_{j}),
\end{align*}
  Moreover, $f_{\tuple{a}, \Neither}(\tuple{x}) = f_{\tuple{a}, \True}(\tuple{x}) \wedge x_{i}$ whenever $a_{i} = \Neither$ and $f_{\tuple{a}, \Both}(\tuple{x}) = f_{\tuple{a}, \True}(\tuple{x}) \wedge x_{i}$ whenever $a_{i} = \Both$.

  If $\{ \Neither, \Both \} \nsubseteq \tuple{a}$, then the smallest partially harmonious function $f_{\tuple{a}, b}$ such that $f_{\tuple{a}, b}(\tuple{a}) = b$ is the smallest harmonious function $f$ such that $f(\tuple{a}) = b$, provided that $\Neither \in \tuple{a}$ if $b = \Neither$ and $\Both \in \tuple{a}$ if $b = \Both$. But we have already proved that these functions belong to $\clonegen{\DMAClone, \Box}$ (see the proof of Theorem~\ref{harmonious clones}).
\end{proof}

  For reasons of space, only the generating sets are shown in Figure~\ref{fig: discriminator clones}. That is, the node labelled $\Delta, \dual$ represents $\clonegen{\Delta, \dual}$ etc.

\begin{theorem}[Discriminator clones above $\DMAClone$] \label{discriminator clones}
  The lattice of De Morgan clones above $\DMAClone$ which contain the De~Morgan discriminator is precisely the finite lattice shown in Figure \ref{fig: discriminator clones}.
\end{theorem}

\begin{proof}
  We have already observed at the beginning of this section that a De~Morgan clone above $\DMAClone$ contains the discriminator function if and only if it contains $\Box x$. Let us first show that the inclusions indicated in the figure hold. We know that $\Box x = x \wedge \dual x = \Delta x \wedge \nabla x$, where $\nabla x = \dmneg \Delta \dmneg x$, and clearly every De~Morgan function can be expressed using $\Delta$, $\Neither$, and $\Both$. Moreover,
\begin{align*}
  \Delta x & = \Box (x \vee \idbton x) = \Box (x \vee \Neither), \\
  \Delta x & = \Diamond (x \wedge \idntob x) = \Diamond (x \wedge \Both), \\
  \idbton x & = (\dmneg \Delta x \wedge x) \vee (\nabla x \wedge x) \vee (\Delta x \wedge \Delta \dmneg x \wedge \Neither), \\
  \idntob x & = (\Delta x \wedge x) \vee (\Delta \dmneg x \wedge x) \vee (\nabla x \wedge \nabla \dmneg x \wedge \Both), \\
  \Deltanbpair(x, y) & = \nabla x \wedge \dmneg \Delta x \wedge \Delta y \wedge \dmneg \nabla y.
\end{align*}

  Secondly, let us show that the figure faithfully depicts meets and joins. To prove the former, it suffices to recall the characterization of clones preserving subalgebras (Theorem~\ref{clones preserving subalgebras}) and the fact that $\clonegen{\DMAClone, \Box}$ is the clone of all harmonious functions preserving $\Btwo$, $\Kthree$, $\Pthree$ (Theorem~\ref{harmonious clones}). To prove the latter, it suffices to observe that $\Neither = \idbton \Both$ and $\Both = \idntob \Neither$ and moreover 
\begin{align*}
  \Delta x & = \Box x \vee \Deltanbpair(\dual x, x), \\
  \dual x & = x \vee (\Delta x \wedge \Delta \dmneg x \wedge \idbton x) \vee (\nabla x \wedge \nabla \dmneg x \wedge \idntob x).
\end{align*}

  Finally, to show that there are no more clones above $\clonegen{\DMAClone, \Box}$, it~suffices to find for each clone $\clone{D}$ in the figure some clones $\clone{D}_{1}$, \dots, $\clone{D}_{n}$ in the figure such that for each De~Morgan clone $\clone{C} \supseteq \clonegen{\DMAClone, \Box}$ we have $\clone{C} \nsubseteq \clone{D}$ if and only if $\clone{D}_{i} \subseteq \clone{C}$ for some $\clone{D}_{i}$. Of course, it suffices to prove this for clones $\clone{D}$ which are meet irreducible in the figure. It will then follow that each proper extension of a clone in the figure lies above one of finitely many covers and all of these covers lie in the lattice shown in the figure. 

  If $\clone{C} \nsubseteq \clonegen{\DMAClone, \Delta, \dual}$, then $\clone{C}$ does not preserve $\Btwo$ by Theorem~\ref{clones preserving subalgebras}, so it contains $\Neither$~or~$\Both$ by Theorem~\ref{clones not preserving subalgebras}. Similarly, if $\clone{C} \nsubseteq \clonegen{\DMAClone, \Box, \Neither}$, then $\clone{C}$ does not preserve $\Kthree$ by Theorem~\ref{clones preserving subalgebras}, so it contains $\Both$ or $\idntob$ or $\dual$ by Theorem~\ref{clones not preserving subalgebras}. By conflation symmetry, if $\clone{C} \nsubseteq \clonegen{\DMAClone, \Box, \Both}$, then it contains $\idbton$ or $\dual$.

 If $\clone{C} \nsubseteq \clonegen{\DMAClone, \dual}$, then $\clone{C}$ is not harmonious, therefore $\clone{C}$ contains one of the functions $\Neither$, $\Both$, $\Truenton$, $\Truebtob$, $\nh_{1}$, \dots, $\nh_{6}$ by Theorem~\ref{non-harmonious clones}. But $\Delta$ is expressible in terms of $\Box$ and any of the functions $\Neither$, $\Both$, $\Truenton$, $\Truebtob$:
\begin{align*}
  \Delta x & = \Box (\Neither \vee x) = \dmneg \Box (\Both \vee \dmneg x) = (x \vee \dmneg x) \wedge \Box \Truenton x = \Box x \vee \dmneg \Box \Truebtob x.
\end{align*}
  Moreover, the binary function $\Deltanbpair$ is expressible in terms of $\Box$ and any of the functions $\nh_{1}$, \dots, $\nh_{6}$ as follows:
\begin{align*}
 \Deltanbpair(x, y) & = \Diamond \nh_{1}(y, x) \\
  & = \Diamond \nh_{2}(y, x) \wedge \dmneg (x \equivcrisp y) \\
  & = \Diamond \nh_{3}(y, x) \wedge \dmneg (x \equivcrisp y) \wedge \Diamond (x \wedge \dmneg x) \\
  & = \Diamond \nh_{4}(y, x) \\
  & = \Diamond \nh_{5}(y, x) \wedge \dmneg (x \equivcrisp y) \\
  & = \Diamond \nh_{6}(y, x) \wedge \dmneg (x \equivcrisp y) \wedge \Diamond (y \wedge \dmneg y).
\end{align*}
  Because $\Deltanbpair$ is expressible in terms of $\Delta$, the clone $\clone{C}$ thus contains $\Deltanbpair$.

  If $\clone{C} \nsubseteq \clonegen{\DMAClone, \Box, \Deltanbpair}$, then either $\clone{C}$ does not preserve $\Kthree$ or $\Pthree$ (in which case it contains $\idbton$ or $\idntob$ or $\dual$ by Theorem~\ref{clones not preserving subalgebras}) or $\clone{C}$ contains a function $f$ which is not partially harmonious. That is, $\dual f(\tuple{a}) \neq f(\dual \tuple{a})$ for some $\tuple{a}$ such that $\{ \Neither, \Both \} \nsubseteq \tuple{a}$. Suppose without loss of generality that $\Both \notin \tuple{a}$ (the other case is analogous). Let us substitute the constant $\True$ for $x_{i}$ if $a_{i} = \True$, the constant $\False$ for $x_{i}$ if $a_{i} = \False$, and the variable $x$ for $x_{i}$ if $a_{i} = \Neither$. This yields either one of the constants $\Neither$ or $\Both$ or a unary non-harmonious function. It~follows by Lemma~\ref{unary non-harmonious functions} that $\clone{C}$ contains $\Neither$ or $\Both$ or $\Truenton$ or $\Truebtob$. But $\Delta$ is expressible in terms of $\Box$ and any of these functions, as shown above.
\end{proof}

\section{Classifying De Morgan clones}
\label{sec: classification}

  We now discuss the properties of expansions of the four-valued Belnap--Dunn logic which correspond to De Morgan clones above $\DMAClone$. In particular, we classify these expansions by their position in the Leibniz and Frege hierarchies of abstract algebraic logic~\cite{font16}.

  One way to introduce the relevant logical notions would be to assign an algebra to each De Morgan clone $\clone{C}$ in a signature which contains an $n$-ary function symbol for each $n$-ary function $f \in \clone{C}$ (or perhaps only for functions in some given generating set of $\clone{C}$), interpreted by the function~$f$. Equipped with the truth predicate $\{ \True, \Both \}$, i.e.\ the truth predicate of Belnap–Dunn logic, this algebra determines a logic. We now call the clone $\clone{C}$ \emph{selfextensional}, \emph{truth-equational}, \emph{protoalgebraic}, or \emph{equivalential} if the corresponding logic is. (Here we disregard Fregean and fully Fregean clones because none of the clones above $\DMAClone$ are Fregean~\cite{font97}. We also disregard fully selfextensional clones because each selfextensional clone over a finite set containing a conjunction is in fact fully selfextensional~\cite{jansana06}.)

  This definition will suffice for readers who are already familiar with these notions in the context of abstract algebraic logic. For other readers, we phrase the definition directly in terms of clones.

\begin{definition}[Clone matrices and logics]
  A \emph{clone matrix} is a pair $\pair{\clone{C}}{F}$ consisting of a clone on a set $A$ and a set $F \subseteq A$. The \emph{logic} $\logic{L}$ determined by the clone matrix $\pair{\clone{C}}{F}$ is a binary relation between subsets of $\clone{C}$, denoted $\Gamma(\tuple{x}) \vdash_{\logic{L}} \Delta(\tuple{x})$, such that $\Gamma(\tuple{x}) \vdash_{\logic{L}} \Delta(\tuple{x})$ if and only if
\begin{align*}
  \Gamma(\tuple{a}) \subseteq F \implies \Delta(\tuple{a}) \subseteq F \text{ for each tuple } \tuple{a} \in A.
\end{align*}
  We use $\Gamma(\tuple{x}) \interdash_{\logic{L}} \Delta(\tuple{x})$ to abbreviate the conjunction of $\Gamma(\tuple{x}) \vdash_{\logic{L}} \Delta(\tuple{x})$ and $\Delta(\tuple{x}) \vdash_{\logic{L}} \Gamma(\tuple{x})$.
\end{definition}

\begin{definition}[Selfextensional logics]
  A logic $\logic{L}$ is called \emph{selfextensional} if $f(\tuple{x}) \interdash_{\logic{L}} g(\tuple{x})$ implies $h(\tuple{x}, f(\tuple{y}), \tuple{z}) \interdash h(\tuple{x}, g(\tuple{y}), \tuple{z})$ for all $f, g, h \in \clone{C}$ such that $f$ and $g$ have the same arity.
\end{definition}

\begin{definition}[Protoalgebraic logics]
  A logic $\logic{L}$ is called \emph{proto\-algebraic} if there is a set $\Delta(x, y) \subseteq \clone{C}$ such that $\emptyset \vdash_{\logic{L}} \Delta(x, x)$ and $x, \Delta(x, y) \vdash_{\logic{L}} y$. Such a set $\Delta$ is called a \emph{protoimplication set}. If $\Delta = \{ f \}$, then the function~$f$ is called a \emph{protoimplication}.
\end{definition}

\begin{definition}[Equivalential logics]
  A logic $\logic{L}$ is called \emph{equivalential} if there is a protoimplication set $\Delta(x, y)$ so that
\begin{align*}
  \Delta(x_{1}, y_{1}), \dots, \Delta(x_{n}, y_{n}) \vdash_{\logic{L}} \Delta(f(x_{1}, \dots, x_{n}), f(y_{1}, \dots, y_{n}))
\end{align*}
  for each $n$-ary function $f$ in $\clone{C}$. Such a set $\Delta$ is called an \emph{equivalence set}. If~$\Delta = \{ f \}$, then the function $f$ is called an \emph{equivalence function}.
\end{definition}

  The definition of truth-equationality is somewhat more complicated, as we have to consider arbitrary models of the logic determined by the clone~$\clone{C}$. (More precisely, as pointed out by the referee, it suffices to consider the $1$-generated models.) We first recall the notion of a \emph{homomorphism of clones} $h\colon \clone{C} \to \clone{C}'$: it is a function which maps $n$-ary functions of $\clone{C}$ to $n$-ary functions of $\clone{C}'$ such that composition of functions is preserved and projections onto the $i$-th component in $\clone{C}$ map to projections onto the $i$-th component in $\clone{C}'$.

\begin{definition}[Models of logics]
  A \emph{model} of the logic $\logic{L}$ of the clone matrix $\pair{\clone{C}}{F}$ is a clone matrix $\pair{\clone{C}'}{F'}$ along with a surjective homomorphism of clones $h\colon \clone{C} \to \clone{C}'$ which preserves the consequence relation: $\Gamma(\tuple{x}) \vdash_{\logic{L}} \Delta(\tuple{x})$ implies $h[\Gamma](\tuple{x}) \vdash_{\logic{L}'} h[\Delta](\tuple{x})$ for each $\Gamma(\tuple{x}), \Delta(\tuple{x}) \subseteq \clone{C}$, where $\logic{L}'$ is the logic determined by the clone matrix $\pair{\clone{C}'}{F'}$.
\end{definition}

  The surjectivity requirement here corresponds to the fact that each function in $\clone{C}'$ is to be thought of as the interpretation of some function in $\clone{C}$.

\begin{definition}[The Leibniz congruence]
  The \emph{Leibniz congruence} of $F \subseteq A$, denoted $\Leibniz{\clone{C}}{F}$, is defined as follows: $\pair{a}{b} \in \Leibniz{\clone{C}}{F}$ if and only if
\begin{align*}
  f(a, \tuple{c}) \in F \iff f(b, \tuple{c}) \in F \text{ for each } f \in \clone{C} \text{ and each tuple } \tuple{c} \in A.
\end{align*}
\end{definition}

\begin{definition}[Truth-equational and algebraizable logics]
  The logic $\logic{L}$ is called \emph{truth-equational} if there is a set of equations (pairs of functions) $t_{i}(x) \approx u_{i}(x)$ for $i \in I$ such that for each model $\pair{\clone{C}}{F}$ of $\logic{L}$ we have
\begin{align*}
  a \in F \iff \pair{t_{i}(a)}{u_{i}(a)} \in \Leibniz{\clone{C}}{F} \text{ for all } i \in I.
\end{align*}
  The logic $\logic{L}$ is called \emph{(weakly) algebraizable} if it is both equivalential (proto\-algebraic) and truth-equational.
\end{definition}

  In the following, we restrict to De~Morgan clones with the designated values $\{ \True, \Both \}$. That is, we take $A \assign \{ \True, \False, \Neither, \Both \}$ and $F \assign \{ \True, \Both \}$. Abusing our terminology somewhat, we call the De~Morgan clone $\clone{C}$ selfextensional (proto\-algebraic, etc.) if the logic determined by $\pair{\clone{C}}{F}$ is selfextensional (proto\-algebraic, etc.). We simply write $\Gamma(\tuple{x}) \vdash \Delta(\tuple{x})$ and $\Gamma(\tuple{x}) \interdash \Delta(\tuple{x})$ if the clone~$\clone{C}$ is clear from the context. That is, $\Gamma(\tuple{x}) \vdash \Delta(\tuple{x})$ if and only if $\Gamma(\tuple{a}) \subseteq \{ \True, \Both \}$ implies $\Delta(\tuple{a}) \subseteq \{ \True, \Both \}$ for each tuple $\tuple{a} \in \DMfour$, where $\Gamma(\tuple{x}), \Delta(\tuple{x}) \subseteq \clone{C}$.

  The general definitions given above can be simplified for De~Morgan clones above $\DMAClone$ as follows. The reader who has not fully digested the definitions above may in effect take these to be our working definitions of self\-extensionality, protoalgebraicity, etc.

\begin{fact}
  A De~Morgan clone above $\DMAClone$ is selfextensional if and only if for all functions $f$ and $g$ in the clone $f \interdash g$ implies $f = g$.
\end{fact}

\begin{proof}
  If $\clone{C} \supseteq \DMAClone$ is selfextensional and $f \interdash g$ for some functions $f$ and $g$ in $\clone{C}$, then $\dmneg f \interdash \dmneg g$ by selfextensionality, hence $f$ and $g$ have the same truth and falsity conditions. By the Truth and Falsity Lemma (Lemma~\ref{truth and falsity lemma}) they are therefore equal. Conversely, if $\clone{C}$ is not selfextensional, then there are $f, g, h \in \clone{C}$ such that $f(\tuple{x}) \interdash g(\tuple{x})$ but not $h(\tuple{x}, f(\tuple{y})) \interdash h(\tuple{x}, g(\tuple{y}))$, hence clearly $f \neq g$.
\end{proof}

\begin{fact}
  A De Morgan clone above $\DMAClone$ is protoalgebraic if and only if it contains a binary function $x \imp y$ such that
\begin{align*}
  & x \imp x \in \{ \True, \Both \} & & \text{and} & & x, x \imp y \in \{ \True, \Both \} \implies y \in \{ \True, \Both \}.
\end{align*}
\end{fact}

\begin{proof}
  Such functions are precisely the De~Morgan protoimplications. Now suppose that a De~Morgan clone above $\DMAClone$ has a protoimplication set $\Delta(x, y)$. For each pair $a \in \{ \True, \Both \}$ and $b \notin \{ \True, \Both \}$ pick some $\delta_{a,b}(x, y) \in \Delta(x, y)$ such that $\delta_{a,b}(a, b) \notin \{ \True, \Both \}$. Such a function exists because $\Delta(x, y)$ is a protoimplication set. Then the conjunction of all these finitely many functions is a protoimplication.
\end{proof}

\begin{fact}
  A De Morgan clone above $\DMAClone$ has is equivalential if and only if it contains a binary function $x \biimp y$ such that $x \biimp y \in \{ \True, \Both \} \iff x = y$.
\end{fact}

\begin{proof}
  Such a function is an equivalence function. Conversely, suppose that a De~Morgan clone above $\DMAClone$ has an equivalence set. As in the previous fact, this implies that it has an equivalence function $x \biimp y$. We may assume that $x \biimp y = y \biimp x$ by taking the equivalence function $(x \biimp y) \wedge (y \biimp x)$ instead of $x \biimp y$. Then $x = y$ implies $x \biimp y \in \{ \True, \Both \}$ by the definition of a protoimplication. Conversely, if $x \neq y$, then we have one of four possibilities:
\begin{enumerate}[(i)]
\item $x \in \{ \True, \Both \}$ and $y \notin \{ \True, \Both \}$,
\item $x \notin \{ \True, \Both \}$ and $y \in \{ \True, \Both \}$,
\item $\dmneg x \in \{ \True, \Both \}$ and $\dmneg y \notin \{ \True, \Both \}$,
\item $\dmneg x \notin \{ \True, \Both \}$ and $\dmneg y \in \{ \True, \Both \}$.
\end{enumerate}
  In (i) and (ii) we have either $x \biimp y \notin F$ or $y \biimp x \notin F$ by the definition of a protoimplication, so $x \biimp y \notin F$ by the assumed symmetry of $\biimp$. In (iii) and (iv) we have either $\dmneg x \biimp \dmneg y \notin F$ or $\dmneg y \biimp \dmneg x \notin F$, so again $\dmneg x \biimp \dmneg y \notin F$. But this implies that $x \biimp y \notin F$ by the definition of an equivalence function.
\end{proof}

  Equivalence functions are not required to be symmetric, but we can always transform a De~Morgan equivalence function $x \biimp y$ into a symmetric equivalence function by taking $(x \biimp y) \wedge (y \biimp x)$ instead.

  Let us now describe the set of all De~Morgan protoimplications. Four important protoimplications are defined in Figure~\ref{fig:min max implication}. They can be described more succinctly as follows: the function $x \imptmax y$ ($x \impimax y$) is the unique De~Morgan function $x \rightarrow y$ with the range $\{ \True, \Neither \}$ (range $\{ \False, \Both \}$) such that
\begin{align*}
  a \rightarrow b \in \{ \True, \Both \} \text{ if and only if } a \in \{ \True, \Both \} \implies b \in \{ \True, \Both \}.
\end{align*}
  The function $x \imptmin y$ ($x \impimin y$), on the other hand, is the unique De Morgan function $x \rightarrow y$ with the range $\{ \False, \Both \}$ (range $\{ \True, \Neither \}$) such that
\begin{align*}
  a \rightarrow b \in \{ \True, \Both \} \text{ if and only if } a = b.
\end{align*}
  Taking the range to be $\{ \True, \False \}$ instead then yields $x \impcrisp y$ and $x \equivcrisp y$.

\begin{figure}
\begin{center}
\begin{tabular}{c | c c c c }
  $\imptmax$ & $\True$ & $\False$ & $\Neither$ & $\Both$ \\
  \hline
  $\True$ & $\True$ & $\Neither$ & $\Neither$ & $\True$ \\
  $\False$ & $\True$ & $\True$ & $\True$ & $\True$ \\
  $\Neither$ & $\True$ & $\True$ & $\True$ & $\True$ \\
  $\Both$ & $\True$ & $\Neither$ & $\Neither$ & $\True$
\end{tabular}
\qquad
\begin{tabular}{c | c c c c }
  $\impimax$ & $\True$ & $\False$ & $\Neither$ & $\Both$ \\
  \hline
  $\True$ & $\Both$ & $\False$ & $\False$ & $\Both$ \\
  $\False$ & $\Both$ & $\Both$ & $\Both$ & $\Both$ \\
  $\Neither$ & $\Both$ & $\Both$ & $\Both$ & $\Both$ \\
  $\Both$ & $\Both$ & $\False$ & $\False$ & $\Both$
\end{tabular}

\bigskip

\begin{tabular}{c | c c c c }
  $\imptmin$ & $\True$ & $\False$ & $\Neither$ & $\Both$ \\
  \hline
  $\True$ & $\Both$ & $\False$ & $\False$ & $\False$ \\
  $\False$ & $\False$ & $\Both$ & $\False$ & $\False$ \\
  $\Neither$ & $\False$ & $\False$ & $\Both$ & $\False$ \\
  $\Both$ & $\False$ & $\False$ & $\False$ & $\Both$
\end{tabular}
\qquad
\begin{tabular}{c | c c c c }
  $\impimin$ & $\True$ & $\False$ & $\Neither$ & $\Both$ \\
  \hline
  $\True$ & $\True$ & $\Neither$ & $\Neither$ & $\Neither$ \\
  $\False$ & $\Neither$ & $\True$ & $\Neither$ & $\Neither$ \\
  $\Neither$ & $\Neither$ & $\Neither$ & $\True$ & $\Neither$ \\
  $\Both$ & $\Neither$ & $\Neither$ & $\Neither$ & $\True$
\end{tabular}
\end{center}
\caption{The minimal and maximal protoimplications}
\label{fig:min max implication}
\end{figure}

  The lattice relations between the four protoimplications of Figure~\ref{fig:min max implication} are analogous to the relations between the elements $\True$, $\False$, $\Neither$, $\Both$, namely:
\begin{align*}
  x \imptmin y & = (x \impimax y) \wedge (x \impimin y), \\
  x \imptmax y & = (x \impimax y) \vee (x \impimin y), \\
  x \impimin y & = (x \imptmax y) \infwedge (x \imptmin y), \\
  x \impimax y & = (x \imptmax y) \infvee (x \imptmin y).
\end{align*}
  In fact, we now show that these four functions form the extreme points of the set of De~Morgan protoimplications in the truth and information orders (under the pointwise ordering of De~Morgan functions).

\begin{theorem}[Smallest and largest protoimplications]
  A binary De Morgan function is a protoimplication if and only if it lies in the interval $[x \imptmin y, x \imptmax y]$ in the truth order, or equivalently if and only if it lies in the interval $[x \impimin y, x \impimax y]$ in the information order.
\end{theorem}

\begin{proof}
  If the function $x \imp y$ belongs to the interval $[x \imptmin y, x \imptmax y]$ in the truth order, then $\True \imp \False \leq \True \imptmax \False \notin \{ \True, \Both \}$ and $\True \imp \False \leq \True \imptmax \False \notin \{ \True, \Both \}$ and $a \imp a \geq a \imptmin a \in \{ \True, \Both \}$, therefore $x \imp y$ is a protoimplication.

  Conversely, suppose that $x \imp y$ is a protoimplication. To prove that $x \imp y \leq x \imptmax y$, it suffices to show that $\True \imp \False \leq \True \imptmax \False = \Neither$ and $\True \imp \Neither \leq \True \imptmax \Neither = \Neither$. But this holds because $\True \imp \False \notin \{ \True, \Both \}$ and $\True \imp \Neither \notin \{ \True, \Both \}$. To prove that $x \imp y \geq x \imptmin y$, it suffices to show that $a \imp a \geq a \imptmin a = \Both$. But this holds because $a \imp a \in \{ \True, \Both \}$.

  The second claim now follows by truth--information symmetry.
\end{proof}

  In addition to the protoimplications introduced above, it is also natural to consider the \emph{G\"{o}del implication} $x \impG y$ such that
\begin{align*}
  a \impG b = \begin{cases} \True & \text{ if } a \leq b, \\  b & \text{ if } a \nleq b.\end{cases}
\end{align*}
  This operation might be also called \emph{Boolean implication}, since it is implication associated with the four-element Boolean lattice reduct of $\DMfouralg$.

\begin{proposition}[Interdefinability of protoimplications] \label{prop: interdefinability of protoimplications}
  The following operations are interdefinable over $\DMAClone$:
\begin{enumerate}[\rm(i)]
\item $x \imptmax y$ and $x \impimin y$,
\item $x \impimax y$ and $x \imptmin y$,
\item $x \equivcrisp y$, $x \impG y$, and $\Box x$,
\item $x \impcrisp y$ and $\Delta x$.
\end{enumerate}
\end{proposition}

\begin{proof}
  (i, ii) The functions $x \impimin y$ and $x \imptmin y$ are definable over $\DMAClone$ in terms of $x \imptmax y$ and $x \impimax y$ by Lemma~\ref{three basic protoimplications}. Conversely,
\settowidth{\auxlength}{$x \imptmax y$}
\begin{align*}
  \hbox to \auxlength{$x \imptmax y$} & = (x \vee \Neither) \impimin (x \vee y \vee \Neither), \\
  \hbox to \auxlength{$x \impimax y$} & = (x \wedge \Both) \imptmin (x \wedge y \wedge \Both),
\end{align*}
  where $\Neither = \True \impimin \False$ and $\Both = \True \imptmin \True$, because
\begin{align*}
  (a \vee \Neither) \impimin (a \vee b \vee \Neither) \in \{ \True, \Both \} \iff a \leq b \vee \Neither, \\
  (a \wedge \Both) \imptmin (a \wedge b \wedge \Both) \in \{ \True, \Both \} \iff a \wedge \Both \leq b.
\end{align*}

  (iii) Clearly $\Box a = \dmneg a \impG \False = \dmneg a \equivcrisp \False$. On the other hand,
\settowidth{\auxlength}{$a \impG b$}
\begin{align*}
  \hbox to \auxlength{$a \impG b$} & = b \vee (a \equivcrisp (a \wedge b)), \\
  \hbox to \auxlength{$a \equivcrisp b$} & = \Box (a \wedge b) \vee \Box (\dmneg a \wedge \dmneg b) \vee (\dmneg \Box (a \vee b) \wedge \dmneg \Box (\dmneg a \vee \dmneg b)).
\end{align*}

  (iv) We have $a \impcrisp b = \dmneg \Delta a \vee \Delta b$ and $\Delta a = \True \impcrisp a$.
\end{proof}

  We now turn to the classification of De Morgan clones above $\DMAClone$ within the Leibniz and Frege hierachies.

\begin{lemma}[Three basic protoimplications] \label{three basic protoimplications}
  Let $\clone{C}$ be a clone above $\DMAClone$ which contains some protoimplication $x \imp y$.
\begin{enumerate}[\rm(i)]
\item If the range of $x \imp y$ is $\{ \True, \Neither \}$, then $\clone{C}$ contains $x \impimin y$.
\item If the range of $x \imp y$ is $\{ \False, \Both \}$, then $\clone{C}$ contains $x \imptmin y$.
\item If the range of $x \imp y$ is $\{ \True, \False \}$, then $\clone{C}$ contains $x \equivcrisp y$.
\end{enumerate}
\end{lemma}

\begin{proof}
  Consider the function
\begin{align*}
  x \imptwo y & \assign {(x \imp y) \wedge (y \imp x) \wedge (\dmneg y \imp \dmneg x) \wedge (\dmneg x \imp \dmneg y)}.
\end{align*}
  Then $x \imptwo y$ is a protoimplication and $a \imptwo b = b \imptwo a = \dmneg b \imptwo \dmneg a = \dmneg a \imptwo \dmneg b$. It follows that $a \imptwo b \notin \{ \True, \Both \}$ whenever either $a \in \{ \True, \Both \}$ and $b \notin \{ \True, \Both \}$ or $b \in \{ \True, \Both \}$ and $a \notin \{ \True, \Both \}$ or $\dmneg b \in \{ \True, \Both \}$ and $\dmneg a \notin \{ \True, \Both \}$ or $\dmneg a \in \{ \True, \Both \}$ and $\dmneg b \notin \{ \True, \Both \}$. Case analysis then shows that $a \imptwo b \notin \{ \True, \Both \}$ whenever $a \neq b$. Conversely, $a \imptwo a = \True$ because $x \imptwo y$ is a protoimplication. The function $x \imptwo y$ is now uniquely determined by its range, which it inherits from $x \imp y$.
\end{proof}

\begin{theorem}[Protoalgebraic clones above $\DMAClone$] \label{protoalgebraic clones}
  A De~Morgan clone above $\DMAClone$ is protoalgebraic if and only if it contains one of the functions $\Box x$, ${x \imptmin y}$, or $x \impimin y$.
\end{theorem}

\begin{proof}
  We have already observed that $x \imptmin y$ and $x \impimin y$ are protoimplications. We also know that $x \equivcrisp y$ is a protoimplication, and that it belongs to the clone $\clonegen{\DMAClone, \Box}$.

  Conversely, suppose that a clone above $\DMAClone$ contains a protoimplication $x \imp y$. If this function does not preserve $\Btwo$, then the clone contains either $\Neither$~or~$\Both$. In the former case, $\Neither \vee (x \imp y)$ is a protoimplication with range $\{ \True, \Neither \}$, therefore the clone contains $x \impimin y$ by Lemma \ref{three basic protoimplications}. In the latter case, $\Both \wedge (x \imp y)$ is a protoimplication with range $\{ \False, \Both \}$, therefore the clone contains $x \imptmin y$ by Lemma \ref{three basic protoimplications}.

  Now suppose that $x \imp y$ preserves $\Btwo$. Consider the function
\begin{align*}
  x \imptwo y & \assign (x \imp y) \vee (x \imp \False) \vee (x \imp (x \imp y)).
\end{align*}
  We first show that this function is a protoimplication. Clearly $a \imptwo a \in \{ \True, \Both \}$, since $a \imptwo a \geq a \imp a \in \{ \True, \Both \}$. Moreover, if $a \in \{ \True, \Both \}$ and $b \notin \{ \True, \Both \}$, then $a \imp \False \notin \{ \True, \Both \}$ and $a \imp b \notin \{ \True, \Both \}$, therefore also $a \imp (a \imp b) \notin \{ \True, \Both \}$. Because $\{ \True, \Both \}$ is a prime filter, it follows that $a \imptwo b \notin \{ \True, \Both \}$.

  Because $x \imp y$ preserves $\Btwo$, so does $x \imptwo y$, therefore $a \imptwo b = a \impcrisp b$ for $a, b \in \{ \True, \False \}$. Moreover, $\Neither \imptwo \False \neq \Neither$: (i) if $\Neither \imp \False = \False$, then $\Neither \imptwo \False = \False$, (ii) if $\Neither \imp \False = \Neither$, then $\Neither \imptwo \False \geq \Neither \imp \Neither = \True$, and (iii) if $\Neither \imp \False \geq \Both$, then $\Neither \imptwo \False \geq \Neither \imp \False \geq \Both$. By the same token $\Both \imptwo \False \neq \Both$. It now follows that $\Box a = a \wedge (\dmneg a \imptwo \False) \wedge \dmneg (a \imptwo \False)$ for each $a \in \DMfour$.
\end{proof}

\begin{corollary}[Equivalential clones above $\DMAClone$]
  A De~Morgan clone above $\DMAClone$ is equivalential if and only if it is protoalgebraic.
\end{corollary}

\begin{proof}
  Each of the functions $x \imptmin y$, $x \impimin y$, and $x \equivcrisp y$ is an equivalence function because $a \imptmin b = \True$ if and only if $a = b$, and likewise for $x \impimin y$ and $x \equivcrisp y$. Recall that $x \equivcrisp y$ lies in $\clonegen{\DMAClone, \Box}$.
\end{proof}

  Let us verify that these three functions indeed generate three distinct minimal protoalgebraic clones above $\DMAClone$.

\begin{fact}
  The clones $\clonegen{\DMAClone, \Box}$, $\clonegen{\DMAClone, \imptmin}$, $\clonegen{\DMAClone, \impimin}$ are pair\-wise incomparable by inclusion.
\end{fact}

\begin{proof}
  The functions $\Box x$ and $x \imptmin y$ preserve $\Pthree$, but $x \impimin y$ does not. Similarly, $\Box x$ and $x \impimin y$ preserve $\Kthree$, but $x \imptmin y$ does not. Finally, consider the binary relation $R \subseteq \DMfour \times \DMfour$ such that $\pair{a}{b} \in R$ if and only if $a, b \in \{ \True, \Neither, \False \}$ and $\pair{a}{b}$ is distinct from both $\pair{\True}{\False}$ and $\pair{\False}{\True}$. Then each function in $\clonegen{\DMAClone, \impimin}$ preserves $R$ but $\Box x$ does not, since $\pair{\True}{\Neither} \in R$ but $\pair{\Box \True}{\Box \Neither} = \pair{\True}{\False} \notin R$. Replacing $\Neither$ by $\Both$ yields a relation which is preserved by each function in $\clonegen{\DMAClone, \imptmin}$ but not by $\Box x$.
\end{proof}

  On the other hand, observe that $\Box x$ lies in $\clonegen{\DMAClone, \imptmin, \impimin}$:
\begin{align*}
  \Box x = x \wedge (\Neither \vee (x \imptmin \True)) \wedge \dmneg (x \impimin \Neither)).
\end{align*}

\begin{theorem}[Truth-equational clones above $\DMAClone$] \label{truth-equational clones}
  A De~Morgan clone above $\DMAClone$ is truth-equational if and only if it contains $\Truenton$ or~$\Truebtob$.
\end{theorem}

\begin{proof}
  If a De~Morgan clone above $\DMAClone$ does not contain any of these functions, then by Lemma~\ref{unary non-harmonious functions} all of its unary functions are harmonious, i.e.\ they commute with conflation. But the set $\{ \True, \Both \}$ is not closed under conflation, therefore it cannot be defined by any equation $f(x) \equals g(x)$ where $f$ and $g$ are harmonious.

  Conversely, we show that if a De~Morgan clone above $\DMAClone$ contains one of these functions, then it is truth-equational. First observe that the set $\{ \True, \Both \}$ is defined by the equation
\begin{align*}
  \Truebtob x & \equals x,
\end{align*}
  as well the pair of equations
\begin{align*}
  \Truenton x \equals \True \text{ and } x \vee \dmneg x \equals x.
\end{align*}
  We claim that these equations witness the truth-equationality of the clone.

  Let us first consider the clone $\clonegen{\DMAClone, \Truebtob}$. Because the equation $\Truebtob x \equals x$ defines the set $\{ \True, \Both \}$, the rules
\begin{align*}
  x, t(\Truebtob x, \tuple{y}) & \vdash t(x, \tuple{y}) & & \text{ and } & x, t(x, \tuple{y}) & \vdash t(\Truebtob x, \tuple{y})
\end{align*}
  are valid in the logic determined by this clone for each term $t(x, \tuple{y})$. It follows that ${a \in F}$ implies $\pair{\Truebtob a}{a} \in \Leibniz{\clone{C}}{F}$ in each model $\pair{\clone{C}}{F}$ of the logic deter\-mined by $\clonegen{\DMAClone, \Truebtob}$. Conversely, suppose that $\pair{\Truebtob a}{a} \in \Leibniz{\clone{C}}{F}$ in some model $\pair{\clone{C}}{F}$ of this logic. We need to prove that $a \in F$. Because the rule $\emptyset \vdash \Truebtob x$ is valid in this logic, we have $\Truebtob a \in F$. But now $a \in F$ because $\pair{\Truebtob a}{a} \in \Leibniz{\clone{C}}{F}$.

  Similarly, consider the clone $\clonegen{\DMAClone}{\Truenton}$. Repeating the argument of the previous paragraph yields that $a \in F$ implies $\pair{\Truenton a}{\True} \in \Leibniz{\clone{C}}{F}$ and $\pair{a \vee \dmneg a}{a} \in \Leibniz{\clone{C}}{F}$ for each model $\pair{\clone{C}}{F}$ of the logic determined by $\clonegen{\DMAClone}{\Truenton}$. Conversely, suppose that both $\pair{\Truenton a}{\True} \in \Leibniz{\clone{C}}{F}$ and ${\pair{a \vee \dmneg a}{a} \in \Leibniz{\clone{C}}{F}}$ hold in some model $\pair{\clone{C}}{F}$ of this logic. We need to prove that $a \in F$. Because the rule $\emptyset \vdash \True$ is valid in this logic and $\pair{\Truenton a}{\True} \in \Leibniz{\clone{C}}{F}$, we may infer $\Truenton a \in F$. But the validity of the rule $\Truenton x \vdash x \vee \dmneg x$ implies that $a \vee \dmneg a \in F$, and finally $\pair{a \vee \dmneg a}{a} \in \Leibniz{\clone{C}}{F}$ implies that $a \in F$.
\end{proof}

\begin{fact}
  $\clonegen{\DMAClone, \Truenton}$ and $\clonegen{\DMAClone, \Truebtob}$ are pair\-wise incomparable by inclusion.
\end{fact}

\begin{proof}
  Consider the binary relation $R \subseteq \DMfour \times \DMfour$ such that $\pair{a}{b} \in R$ if and only if either $a = \Both$ and $b \in \{ \True, \Neither, \False \}$ or $a = \True = b$ or $a = \False = b$. Then each function in $\clonegen{\DMAClone, \Truebtob}$ preserves $R$ but $\Truenton$ does not. Switching $\Both$ and $\Neither$ yields a relation which preserved by $\clonegen{\DMAClone, \Truenton}$ but not by $\Truebtob$.
\end{proof}

\begin{corollary}[Algebraizable clones above $\DMAClone$]
  A De~Morgan clone above $\DMAClone$ is algebraizable (or equivalently, weakly algebraizable) if and only if it contains one of the functions $\Delta x$, $x \imptmin y$, $x \impimin y$.
\end{corollary}

\begin{proof}
  Recall that a clone is called (weakly) algebraizable if it is truth-equational and equivalential (protoalgebraic). By Theorems \ref{protoalgebraic clones}~and~\ref{truth-equational clones}, a clone above $\DMAClone$ is algebraizable if and only if it contains one of the \mbox{functions} $\{ \Box, \imptmin, \impimin \}$ and one of the functions $\{ \Truenton, \Truebtob \}$. If the clone contains $\imptmin$, then it contains $\Truebtob$ because $\Truebtob x = x \vee \dmneg x \vee (\True \imptmin \True)$. Similarly, if it contains $\impimin$, then it already contains $\Truenton$ because $\Truenton x = x \vee \dmneg x \vee (\True \impimin x)$. A clone above $\DMAClone$ is thus algebraizable if and only if it contains $\imptmin$ or $\impimin$ or both $\Box$ and either $\Truenton$ or $\Truebtob$. But
\begin{align*}
  \Delta x & = \Box x \vee (\Box \Truenton x \wedge \dmneg \Box (x \vee \dmneg x)) = \Box x \vee \dmneg \Box \Truebtob x.
\end{align*}
  Alternatively, we may observe that $\Truenton$ and $\Truebtob$ preserve $\Btwo$, $\Kthree$, and $\Pthree$, but they are not partially harmonious. By Theorem~\ref{discriminator clones}, this means that $\clonegen{\DMAClone, \Box, \Truenton} = \clonegen{\DMAClone, \Box, \Truebtob} = \clonegen{\DMAClone, \Delta}$.
\end{proof}

\begin{fact}
  The clones $\clonegen{\DMAClone, \Delta}$, $\clonegen{\DMAClone, \imptmin}$, $\clonegen{\DMAClone, \impimin}$ are pair\-wise incomparable by inclusion.
\end{fact}

\begin{proof}
  The function $\Delta x$ preserves $\Btwo$, while the other two functions do not. On the other hand, we have already shown that neither $\clonegen{\DMAClone, \imptmin}$ nor $\clonegen{\DMAClone, \impimin}$ lies above $\clonegen{\DMAClone, \Box}$, therefore neither clone lies above $\clonegen{\DMAClone, \Delta}$.
\end{proof}

\begin{corollary}[Protoalgebraic non-truth-eq.\ clones above $\DMAClone$]
  The only proto\-algebraic clones lying above $\DMAClone$ which are not truth-equational are $\clonegen{\DMAClone, \Box}$, $\clonegen{\DMAClone, \dual}$, and $\clonegen{\DMAClone, \Box, \Deltanbpair}$.
\end{corollary}

\begin{proof}
  Each protoalgebraic clone either contains $\Box x$ or it contains one of the functions $x \imptmax y$ or $x \impimax y$ (Theorem~\ref{protoalgebraic clones}). But $\Neither = {\True \imptmax \False}$ and $\Both = \True \impimax \True$ and each clone above $\DMAClone$ which contains $\Neither$ or $\Both$ is truth-equational (Theorem~\ref{truth-equational clones}). Therefore each protoalgebraic clone above $\DMAClone$ is truth-equational unless it contains $\Box$. But the only clones among the extensions of $\clonegen{\DMAClone, \Box}$ which are not truth-equational are $\clonegen{\DMAClone, \Box}$, $\clonegen{\DMAClone, \dual}$, and $\clonegen{\DMAClone, \Box, \Deltanbpair}$ by Theorems \ref{discriminator clones}~and~\ref{truth-equational clones}.
\end{proof}

\begin{theorem}[Selfextensional clones above $\DMAClone$] \label{selfextensional clones}
  A De Morgan clone above $\DMAClone$ is selfextensional if and only if it lies below the clone $\clonegen{\DMAClone, \dual}$, i.e.\ if and only if it only contains harmonious functions.
\end{theorem}

\begin{proof}
  Suppose that a clone lies between $\DMAClone$ and $\clonegen{\DMAClone, \dual}$. Then this clone is harmonious by Theorem~\ref{positive clone and persistent clone}. We show that it is selfextensional. If $f$ and $g$ are harmonious functions and $f \interdash g$, then $f$ and $g$ have the same truth conditions, therefore by harmonicity also the same non-falsity conditions. But then $f = g$ by the Truth and Falsity Lemma (Lemma~\ref{truth and falsity lemma}).

  Conversely, suppose that a clone above $\DMAClone$ is not below $\clonegen{\DMAClone, \dual}$. Then it contains either $\mnh_{1}$ or $\mnh_{2}$ by Theorem~\ref{non-harmonious clones}. (Recall Figure~\ref{fig: non-harmonious binary functions}.) It suffices to find functions $f$ and $g$ in the given clone such that $f \interdash \mnh_{1}$ but $f \neq \mnh_{1}$, and similarly for $g$ and $\mnh_{2}$. In particular, it suffices to find a function $f$ and $g$ in $\DMAClone$ such that $f \interdash \mnh_{1}$ and $g \interdash \mnh_{2}$, since $f$ and $g$ are then harmonious while $\mnh_{1}$ and $\mnh_{2}$ are not. For this purpose it suffices to take
\begin{align*}
  f(x, y) & \assign y \wedge \dmneg y \wedge (x \vee \dmneg x), & g(x, y) & \assign y \wedge \dmneg y. \qedhere
\end{align*}
\end{proof}

\begin{corollary}[Protoalgebraic selfextensional clones above $\DMAClone$]
  The only protoalgebraic selfextensional clones above $\DMAClone$ are the clones $\clonegen{\DMAClone, \Box}$ and $\clonegen{\DMAClone, \dual}$. These are equivalential but not algebraizable.
\end{corollary}

  These in a sense correspond to the two best behaved expansions of Belnap–Dunn logic. Both of these expansions in fact already been studied. The expansion by the operator $\Box$ was studied by Font \& Rius~\cite{font+rius00}, while the expansion by Boolean negation was studied by De \& Omori~\cite{de+omori15}. (Recall that conflation is interdefinable with Boolean negation over $\DMAClone$.)

  Finally, let us observe for the sake of completeness that none of the three covers of $\DMAClone$ (recall Theorem~\ref{covers of dma}) are protoalgebraic or truth-equational.

\subsection*{Acknowledgement} The author is grateful to the anonymous referee for their careful reading of the manuscript and useful comments, which helped to improve the paper.

\end{document}